\documentclass[]{amsart}

\usepackage[english]{babel}
\usepackage[utf8]{inputenc}
\usepackage{microtype}
\usepackage{paralist}
\usepackage{amsmath, amsfonts, amssymb, amsthm}
\usepackage{color}
\usepackage[normalem]{ulem}
\usepackage{tikz}
\usepackage{float}
\usepackage[labelfont=sc]{caption}
\usepackage{listings}
\usepackage{tabularx}
\usepackage{nameref}
\usepackage{hyperref}

\hypersetup{colorlinks,
	linkcolor=green!50!black,
	citecolor=blue!70!black,
	urlcolor=red!50!black
}

\restylefloat{table}

\newcommand{\NN}{\mathbb{N}}

\newcommand{\QQ}{\mathbb{Q}}

\newcommand{\ZZ}{\mathbb{Z}}

\newcommand{\B}{\mathcal{B}}
\newcommand{\C}{\mathcal{C}}

\newcommand{\Pfr}{\mathfrak{P}}

\newcommand{\set}[1]{\left\{ #1 \right\}}
\newcommand{\setb}[1]{\left( #1 \right)}
\newcommand{\abs}[1]{\left| #1 \right|}

\newcommand{\step}[2]{\smallskip\noindent \hypertarget{#1}{\textbf{#2}}}
\newcommand{\ts}[1]{\textsuperscript{#1}}

\newtheorem{mymasterthm}{notForUse}%[section]
\theoremstyle{definition}

\newtheorem{myrem}[mymasterthm]{Remark}

\theoremstyle{plain}
\newtheorem{mylemma}[mymasterthm]{Lemma}
\newtheorem{mythm}[mymasterthm]{Theorem}

\newtheorem{myprop}[mymasterthm]{Proposition}

\newtheorem{myproblem}[mymasterthm]{Problem}

\newtheoremstyle{Step}{5pt}{\topsep}{\normalfont}{0pt}{\bfseries}{:}{5pt plus 1pt minus 1pt}{}
\newtheoremstyle{ReductionStep}{5pt}{\topsep}{\normalfont}{0pt}{\bfseries}{:}{5pt plus 1pt minus 1pt}{}

\theoremstyle{Step}
\newtheorem{mystep}{Step} 

\theoremstyle{ReductionStep}
\newcounter{countRedSteps}

\newtheorem{redstep}[countRedSteps]{Reduction Step}

\allowdisplaybreaks

\title[On the Diophantine equation $ U_n - b^m = c $]{On the Diophantine equation $ U_n - b^m = c $}
\makeatletter
\@namedef{subjclassname@2020}{\textup{2020} Mathematics Subject Classification}
\makeatother

\author[S. Heintze]{Sebastian Heintze}
\address{Sebastian Heintze\newline
	\indent Graz University of Technology\newline
	\indent Institute of Analysis and Number Theory\newline
	\indent Steyrergasse 30/II \newline
	\indent A-8010 Graz, Austria}
\email{heintze@math.tugraz.at}

\author[R.F. Tichy]{Robert F. Tichy}
\address{Robert F. Tichy\newline
	\indent Graz University of Technology\newline
	\indent Institute of Analysis and Number Theory\newline
	\indent Steyrergasse 30/II \newline
	\indent A-8010 Graz, Austria}
\email{tichy@tugraz.at}

\author[I. Vukusic]{Ingrid Vukusic}
\address{Ingrid Vukusic\newline
	\indent University of Salzburg\newline
	\indent Department of Mathematics\newline
	\indent Hellbrunnerstr. 34 \newline
	\indent A-5020 Salzburg, Austria}
\email{ingrid.vukusic@plus.ac.at}

\author[V. Ziegler]{Volker Ziegler}
\address{Volker Ziegler\newline
	\indent University of Salzburg\newline
	\indent Department of Mathematics\newline
	\indent Hellbrunnerstr. 34 \newline
	\indent A-5020 Salzburg, Austria}
\email{volker.ziegler@plus.ac.at}

\thanks{This work was supported by the Austrian Science Fund (FWF) under the project I4406.}

\begin{document}

	\begin{abstract}
 Let $(U_n)_{n\in \NN}$ be a fixed linear recurrence sequence defined over the integers (with some technical restrictions). We prove that there exist effectively computable constants $B$ and $N_0$ such that for any $b,c\in \ZZ$ with $b> B$ the equation $U_n - b^m = c$ has at most two distinct solutions $(n,m)\in \NN^2$ with $n\geq N_0$ and $m\geq 1$. Moreover, we apply our result to the special case of Tribonacci numbers given by $T_1= T_2=1$, $T_3=2$ and $T_{n}=T_{n-1}+T_{n-2}+T_{n-3}$ for $n\geq 4$. By means of the LLL-algorithm and continued fraction reduction we are able to prove $N_0=1.1\cdot 10^{37}$ and $B=e^{438}$. The corresponding reduction algorithm is implemented in Sage.
	\end{abstract}
	
	\maketitle
	
%	{
%		\noindent
%		\textcolor{red}{working version \hfill \today}
%	}

	\section{Introduction}
	
In the last couple of years investigating Pillai-type problems with linear recurrence sequences has been very popular (see Table \ref{table:refs}). 
\begin{table}[h]
\caption{Overview of results on $U_n-V_m=c$}\label{table:refs}
\begin{tabular}{p{0.3\textwidth} p{0.3\textwidth} p{0.3\textwidth}}
\hline
$U_n$                                       & $V_m$                      & authors                                                       \\ \hline
Fibonacci numbers                           & powers of 2                & Ddamulira, Luca, Rakotomalala \cite{DdamuliraLucaRakotomalala2017} \\
Fibonacci numbers                           & Tribonacci numbers         & Chim, Pink, Ziegler \cite{ChimPinkZiegler2017}                     \\
Tribonacci numbers                          & powers of 2                & Bravo, Luca, Yaz\'{a}n \cite{BravoLucaYazan2017}                    \\
$k$-Fibonacci number                        & powers of 2                & Ddamulira, G\'{o}mez, Luca \cite{DdamuliraGomezCarlosLuca2018}     \\
Pell numbers								& powers of 2                & Hernane, Luca, Rihane, Togb\'{e} \cite{HernaneLucaRihaneTogbe2018} \\
Tribonacci numbers                          & powers of 3                & Ddamulira \cite{Ddamulira2019Tribos}                               \\
Fibonacci numbers                           & Pell numbers               & Hern\'{a}ndez, Luca, Rivera \cite{HernandezLucaRivera2019}          \\
Padovan numbers	 							& powers of 2                & Lomel\'{\i}, Hern\'{a}ndez \cite{LomeliHernandez2019} \\
Padovan numbers								& powers of 3                & Ddamulira \cite{Ddamulira2019}                                     \\
Padovan numbers                             & Tribonacci numbers         & Lomel\'{\i}, Hern\'{a}ndez, Luca \cite{LomeliHernandezLuca2019Indian} \\
Fibonacci numbers							& Padovan numbers		     & Lomel\'{\i}, Hern\'{a}ndez, Luca \cite{LomeliHernandezLuca2019} \\
Fibonacci numbers							& powers of 3                & Ddamulira \cite{Ddamulira2020} \\
$k$-Fibonacci numbers                       & powers of 3                & Ddamulira, Luca \cite{DdamuliraLuca2020}                           \\
$X$-coordinates of Pell equations           & powers of 2                & Erazo, G\'{o}mez, Luca \cite{ErazoGomezLuca2021}                    \\
$k$-Fibonacci numbers                       & Pell numbers               & Bravo, D\'{\i}az, G\'{o}mez \cite{BravoDiazGomez2021}               \\ \hline
\end{tabular}
\end{table}
This trend was started in 2017 by Ddamulira, Luca, Rakotomalala \cite{DdamuliraLucaRakotomalala2017}, who proved that the only integers $c$ having at least two representations of the form $F_n - 2^m$ are $c\in \{0,1,-1,-3,5,-11, -30,85\}$ (here $F_n$ is the $n$-th Fibonacci number). 
This problem was inspired by a result due to S. S. Pillai. In 1936 Pillai \cite{Pillai1936, Pillai1937} proved that if $a$ and $b$ are coprime integers, then there exists a constant $c_0(a,b)$ depending on $a$ and $b$ such that for any $c>c_0(a,b)$ the equation 
\begin{equation}\label{eq:Pillai}
	a^n-b^m=c
\end{equation}
has at most one solution $(n,m)\in \ZZ_{>0}^2$.   
A natural generalisation of this problem is to replace $a^n$ and $b^m$ by other linear recurrence sequences. This is what the authors in \cite{DdamuliraLucaRakotomalala2017} did, and also what all the other authors in Table~\ref{table:refs} have done. All these results use lower bounds for linear forms in logarithms and reduction methods. Moreover, there exists a general result:
Chim, Pink and Ziegler \cite{ChimPinkZiegler2018} proved that for two fixed linear recurrence sequences $(U_n)_{n\in \NN}$, $(V_n)_{n\in \NN}$ (with some restrictions) the equation
\begin{equation*}
	U_n - V_m = c
\end{equation*}
has at most one solution $(n,m)\in \ZZ_{>0}^2$ for all $c\in \ZZ$, 
except if $c$ is in a finite and effectively computable set $\C \subset \ZZ$ that depends on $(U_n)_{n\in \NN}$ and $(V_n)_{n\in \NN}$.

In this paper, we would like to generalize that result by ``unfixing'' one of the linear recurrence sequences. 
In the classical setting, it is possible to ``unfix'' $a$ and $b$ completely: Bennett \cite{Bennett2001} proved that for any integers $a,b\geq 2$ and $c\geq 1$ Equation~\eqref{eq:Pillai} has at most two solutions $(n,m)\in \ZZ_{>0}^2$.
Moreover, he conjectured that in fact the equation has at most one solution $(n,m)$ for all but 11 specific exceptional triples $(a,b,c)$.

Of course, we cannot simply say that $(U_n)_{n\in \NN}$ and $(V_n)_{n\in \NN}$ should be completely arbitrary. However, there already exist results where the linear recurrence sequence $(U_n)_{n\in \NN}$ is not entirely fixed: In Table~\ref{table:refs} there are some results involving $k$-Fibonacci numbers \cite{DdamuliraLuca2020, BravoDiazGomez2021}, where $k$ is variable. 
Now what we will do is fix $(U_n)_{n\in \NN}$ and let $V_m=b^m$ with variable $b$. Our main result will be that for fixed $(U_n)_{n\in \NN}$ (with some restrictions) the equation 
\[
	U_n - b^m = c
\]
has at most two distinct solutions $(n,m) \in \ZZ_{>0}^2$ for any $(b,c)\in \ZZ^2$ with only finitely many exceptions $b \in \B$, where $\B$ is an effectively computable set. Allowing two solutions (instead of one solution) is the price we have to pay for letting $b$ vary. The second solution is needed for technical reasons, but we believe that the result might also be true if we only allow at most one solution.
Finally, note that our method does not enable us to solve the problem for a specific sequence $(U_n)_{n\in \NN}$ completely. We will show how far we can get by computing the effective bounds for the Tribonacci numbers and reducing the bounds as far as possible.

Let us outline the rest of this paper. The next section contains some notations  and our results: Theorem~\ref{thm:mainthm} is the main theorem, Theorem~\ref{thm:Tribos} shows what happens if we apply our methods to the Tribonacci numbers. Moreover, we make several remarks on the assumptions in Theorem~\ref{thm:mainthm} and pose some open problems regarding Theorem~\ref{thm:Tribos}. Section~\ref{sec:diophApprox} is a collection of rather well known results from Diophantine approximation. 
Section~\ref{sec:proofMainThm} is devoted to the proof of Theorem~\ref{thm:mainthm} and Section~\ref{sec:Tribos} is devoted to the proof of Theorem~\ref{thm:Tribos}.
Beforehand, in Section~\ref{sec:overviewProof}, we give an overview of the two proofs. In particular, we point out the parallels and differences between the two proofs.

	\section{Notation and results}
	
	A linear recurrence sequence $ (U_n)_{n \in \NN} $ is given by finitely many initial values together with a recursive formula of the shape
	\begin{equation*}
		U_{n+\ell} = w_{\ell-1} U_{n+\ell-1} + \cdots + w_0 U_n.
	\end{equation*}
	We say that such a recurrence sequence is defined over the integers if the coefficients $ w_0, \ldots, w_{\ell-1} $ as well as the initial values are all integers.
	In this situation all elements of the sequence are integers.
	It is well known that any such linear recurrence sequence can be written in its Binet representation
	\begin{equation*}
		U_n = a_1(n) \alpha_1^n + \cdots + a_k(n) \alpha_k^n,
	\end{equation*}
	where the characteristic roots $ \alpha_1, \ldots, \alpha_k $ are algebraic integers and the coefficients $ a_1(n), \ldots, a_k(n) $ are polynomials in $ n $ with coefficients in $ \QQ(\alpha_1, \ldots, \alpha_k) $.
	The recurrence sequence is called simple if $ a_1(n), \ldots, a_k(n) $ are all constant, i.e.\ independent of $ n $.
	Moreover, $ \alpha_1 $ is called the dominant root if $ \abs{\alpha_1} > \abs{\alpha_i} $ for all $ i = 2, \ldots, k $.
	Our result is now the following theorem:
	
	\begin{mythm}
		\label{thm:mainthm}
		Let $ (U_n)_{n \in \NN} $ be a simple linear recurrence sequence defined over the integers with Binet representation
		\begin{equation*}
			U_n = a \alpha^n + a_2 \alpha_2^n + \cdots + a_k \alpha_k^n
		\end{equation*}
		and irrational dominant root $ \alpha > 1 $.
		Assume further that $ a > 0 $, that $ a $ and $ \alpha $ are multiplicatively independent, and that the equation
		\begin{equation}
			\label{eq:techcond}
			\alpha^z - 1 = a^x \alpha^y
		\end{equation}
		has no solutions with $ z \in \NN $, $ x,y \in \QQ $ and $ -1 < x < 0 $.
		Then there exist effectively computable constants $ B \geq 2 $ and $ N_0 \geq 2 $ such that the equation
		\begin{equation}
			\label{eq:centraleq}
			U_n - b^m = c
		\end{equation}
		has for any integer $ b > B $ and any $ c \in \ZZ $ at most two distinct solutions $ (n,m) \in \NN^2 $ with $ n \geq N_0 $ and $ m \geq 1 $.
	\end{mythm}
	
	Let us give some remarks regarding the technical condition involving Equation~\eqref{eq:techcond} in the above theorem:
	
	\begin{myrem}
		The technical condition containing Equation \eqref{eq:techcond}
		can be effectively checked for any given recurrence sequence $ (U_n)_{n \in \NN} $:
		
		First note that by construction $ \alpha $ is an algebraic integer.
		Moreover, note that the ideals $ (\alpha) $ and $ (\alpha^z-1) $ with $ z \in \NN $ have no common prime ideals in their factorisations.
		
		Let $ \Pfr_1, \ldots, \Pfr_n $ be the prime ideals that appear in the prime ideal factorisation of $ (a) $. 
		If $ \Pfr_i $ is not a prime factor of $ (\alpha) $, then let $ k_i $ be the order of $ \alpha $ modulo $ \Pfr_i $, i.e.\ $ k_i $ is minimal such that $ \Pfr_i $ is a prime factor of $ (\alpha^{k_i}-1) $.
		Note that if $ \Pfr_i $ lies above $ (p_i) $ and $ f_i $ is the inertia degree, then $ k_i \mid p_i^{f_i}-1 $, so the $k_i$ are bounded.
		Thus we can compute the maximum of all these orders $ k_0 := \max k_i $.
		
		By Schinzel's theorem on primitive divisors \cite{Schinzel1974}, there exists an effectively computable number $ n_0 $ such that $ \alpha^z - 1 $ has a primitive divisor for any $ z > n_0 $.
		This means that for $ z > \max \set{k_0, n_0} $ the ideal $ (\alpha^z-1) $ has a primitive divisor which is not a divisor of $ (a) $.
		Since $ (\alpha) $ and $ (\alpha^z-1) $ have no common divisors, it is impossible that $ \alpha^z - 1 = a^x \alpha^y $ for $ z > \max \set{k_0,n_0} $. 
		
		For each $ z = 1, \ldots, \max \set{k_0,n_0} $ one can check whether $ \alpha^z - 1 = a^x \alpha^y $ has a solution with $ x,y \in \QQ $ and $ -1 < x < 0 $ by looking at the primes of $ \alpha^z-1 $, $ a $ and $ \alpha $.
	\end{myrem}
	
	\begin{myrem}
		\label{rem:easierCond}
		Let $ \Pfr_1, \ldots, \Pfr_n $ be all prime ideals that appear in the prime ideal factorisations of $ (a) $ and $ (\alpha) $.
		Then we can write
		\begin{align*}
			(a) &= \Pfr_1^{a_1} \cdots \Pfr_n^{a_n}, \\
			(\alpha) &= \Pfr_1^{b_1} \cdots \Pfr_n^{b_n},
		\end{align*}
		where the $ a_i $ and $ b_i $ are integers.
		The following two conditions are relatively easy to check and each of them implies the technical condition containing Equation \eqref{eq:techcond}:
		\begin{enumerate}[I)]
			\item
			\label{it:condDet}
			There are $ i,j \in \set{1,\ldots,n} $ such that
			\begin{equation*}
				\det
				\begin{pmatrix}
					a_i & b_i \\
					a_j & b_j 
				\end{pmatrix}
				= \pm 1.
			\end{equation*}
			\item
			\label{it:condUnit}
			$ \alpha $ is a unit and there is an index $ i \in \set{1,\ldots,n} $ with $ a_i = \pm 1 $.
		\end{enumerate}
	\end{myrem}
	
	\begin{proof}
		If \eqref{eq:techcond} is satisfied, then the factorisation of $ (\alpha^z-1) $ contains also only the prime ideals $ \Pfr_1, \ldots, \Pfr_n $ and we can write
		\begin{equation*}
			(\alpha^z-1) = \Pfr_1^{z_1} \cdots \Pfr_n^{z_n}
			= (\Pfr_1^{a_1} \cdots \Pfr_n^{a_n})^x (\Pfr_1^{b_1} \cdots \Pfr_n^{b_n})^y,
		\end{equation*}
		which implies 
		\begin{equation*}
			z_i = a_i x + b_i y
		\end{equation*}
		for $ i = 1, \ldots, n $.
		Note that all the $ z_i, a_i, b_i $ are integers.
		Therefore if \ref{it:condDet}) is satisfied, then it follows that $ x $ and $ y $ are integers as well and in particular we do not have $ -1 < x < 0 $.
		If \ref{it:condUnit}) is satisfied, then $ b_1 = \cdots = b_n = 0 $, so $ a_i = \pm 1 $ implies that $ x $ is an integer and again we do not have $ -1 < x < 0 $.
	\end{proof}
	
	\begin{myrem}
		Theorem \ref{thm:mainthm} can be applied to the Fibonacci numbers.
		Here we have $ a = \frac{1}{\sqrt{5}} $ and $ \alpha = \frac{1+\sqrt{5}}{2} $, i.e.\ $ \alpha $ is a unit and $ (a) = (\sqrt{5})^{-1} $. Thus condition \ref{it:condUnit}) in Remark \ref{rem:easierCond} is satisfied.
	\end{myrem}
	
	\begin{myrem}
		Theorem \ref{thm:mainthm} can be applied to the linear recurrence sequence given by $ U_0 = 0 $, $ U_1 = 1 $ and $ U_{n+2} = U_{n+1} + 3 U_n $ for $ n \geq 0 $.
		Here we have $ a = \frac{1}{\sqrt{13}} $ and $ \alpha = \frac{1+\sqrt{13}}{2} $, i.e.\ $ (\alpha) = \left( \frac{1+\sqrt{13}}{2} \right)^1 $ is prime (it lies over $ p=3 $) and $ (a) = (\sqrt{13})^{-1} $. Thus condition \ref{it:condDet}) in Remark \ref{rem:easierCond} is satisfied.
	\end{myrem}
	
	\begin{myrem}
		If we weaken Theorem \ref{thm:mainthm} in the sense that we prove the existence of at most three solutions, then an inspection of the proof shows that the technical condition containing Equation \eqref{eq:techcond} is not needed any more.
		
		Furthermore, it is not clear, whether all assumptions in Theorem \ref{thm:mainthm} are really necessary for the statement to be true or only required for our proof to work.
	\end{myrem}
	
As a special case of Theorem \ref{thm:mainthm} we get the following result for the Tribonacci numbers, where the technical condition is checked directly in the proof (Section \ref{sec:Tribos}). Note that the case of Fibonacci numbers has recently been considered by Batte et al.~\cite{BatteEtAl2022}.

\begin{mythm}\label{thm:Tribos}
Let $(T_n)_{n\in \NN}$ be the Tribonacci sequence given by  $T_1=1, T_2=1, T_3=2$ and $T_{n}=T_{n-1}+T_{n-2}+T_{n-3}$ for $n\geq 4$.
If for some integers $b\geq 2$ and $c$ the equation $T_n-b^m=c$ has at least three solutions in positive integers $n,m$ given by $(n_1,m_1),(n_2,m_2),(n_3,m_3)$ with $n_1>n_2>n_3\geq 2$, then 
\[
	\log b \leq 438
	\quad \text{and} \quad 
	150<n_1\leq 1.1\cdot 10^{37}.
\]
\end{mythm}	

\begin{myrem}
The assumption $n_1>n_2>n_3\geq 2$ in the above theorem is natural because of $T_1=T_2$.
\end{myrem}

In view of this result, the following question remains:

\begin{myproblem}\label{problem:Tribos-complete}
Do there exist any pairs $(b,c)$ such that the equation $T_n-b^m=c$ has at least three solutions? 
\end{myproblem}

Moreover, in the proof of Theorem \ref{thm:Tribos} (Section \ref{sec:Tribos}, \hyperlink{step:smallSols}{``small solutions''}) we will search for $b$'s and $c$'s such that $T_n-b^m=c$ has at least two small solutions. We will only find two solutions for
\begin{equation}
	\label{eq:bc}
	\begin{split}
		(b,c) \in	
		\{&(2, -8), (2, -3), (2, -1), (2, 0), (2, 5), (3, -2), (3, 4), (5, -121), \\
		&(5, -1), (5, 19), (7, -5), (17, -15), (54, 220), (641, -137)\}.
	\end{split}
\end{equation}
Thus the following question remains:

\begin{myproblem}
Except for the pairs from \eqref{eq:bc}, do there exist any further $(b,c)$ such that $T_n-b^m=c$ has at least two solutions?
\end{myproblem}

\begin{myrem}
In the proof of Theorem \ref{thm:Tribos} (Section \ref{sec:Tribos}, \hyperlink{step:smallSols}{``small solutions''}) we search for all $2\leq n_2 < n_1 \leq 150$ such that the difference of the corresponding Tribonacci numbers can be written in the form $T_{n_1}-T_{n_2}=b^{m_1}-b^{m_2}$ with $b\geq 2$ and $m_1>m_2\geq 1$. 
The bound 150 is chosen because the computations only take a few minutes and the proof of the upper bound in Theorem \ref{thm:Tribos} is easier if we assume $n_1>150$.
In fact, the authors also ran the computations further and checked if there are $2\leq n_2 < n_1 \leq 350$ such that $T_{n_1}-T_{n_2}=b^{m_1}-b^{m_2}$. These computations took about a week on a usual computer using 4 cores. No further solutions than those in \eqref{eq:bc} were found. At this point, the computations start taking  pretty long because the factorisation of huge $T_{n_1}-T_{n_2}$ is expensive.
\end{myrem}	
	
	\section{Results from Diophantine Approximation}\label{sec:diophApprox}
In this section we state all results from Diophantine approximation, that will be used in the proofs below. In particular, we will use lower bounds for linear forms in logarithms, i.e.\ Baker-type bounds for expressions of the form $|\Lambda|=|b_1 \log \eta_1 + \dots + b_t \log \eta_t|$. These linear forms will be coming from expressions of the form $|\eta_1^{b_1}\cdots \eta_t^{b_t}-1|$ and we will switch between these expressions via the following lemma.

	\begin{mylemma}
		\label{lemma:linsmall}
		Let $ \lambda $ be a real number with $ \abs{\lambda} \leq 1 $.
		Then we have the inequality
		\begin{equation*}
			\frac{1}{4} \abs{\lambda} \leq \abs{e^{\lambda} - 1} \leq 2 \abs{\lambda}.
		\end{equation*}
	\end{mylemma}
	
	\begin{proof}
		The proof for these bounds is implied by a straight-forward calculation.
		For the upper bound we have
		\begin{align*}
			\abs{e^{\lambda} - 1} &= \abs{\sum_{t=1}^{\infty} \frac{\lambda^t}{t!}}
			\leq \sum_{t=1}^{\infty} \frac{\abs{\lambda}^t}{t!}
			= \abs{\lambda} \cdot \sum_{t=0}^{\infty} \frac{\abs{\lambda}^t}{(t+1)!} \\
			&\leq \abs{\lambda} \cdot \sum_{t=0}^{\infty} \frac{1}{(t+1)!}
			= \abs{\lambda} \cdot (e-1) \leq 2 \abs{\lambda}
		\end{align*}
		and for the lower bound we have
		\begin{align*}
			\abs{e^{\lambda} - 1} &= \abs{\sum_{t=1}^{\infty} \frac{\lambda^t}{t!}}
			\geq \abs{\lambda} - \abs{\sum_{t=2}^{\infty} \frac{\lambda^t}{t!}}
			\geq \abs{\lambda} - \sum_{t=2}^{\infty} \frac{\abs{\lambda}^t}{t!} \\
			&= \abs{\lambda} - \abs{\lambda} \cdot \sum_{t=1}^{\infty} \frac{\abs{\lambda}^t}{(t+1)!}
			\geq \abs{\lambda} - \abs{\lambda} \cdot \sum_{t=1}^{\infty} \frac{1}{(t+1)!} \\
			&= \abs{\lambda} \cdot (3-e) \geq \frac{1}{4} \abs{\lambda}.
		\end{align*}
	\end{proof}

Also, we will use the following simple fact.

	\begin{mylemma}
		\label{lemma:linbig}
		Let $ \lambda $ be a real number with $ \abs{\lambda} > 1 $.
		Then we have the inequality
		\begin{equation*}
			\abs{e^{\lambda} - 1} \geq \frac{3}{5}.
		\end{equation*}
	\end{mylemma}
	
	\begin{proof}
		Since $ \abs{\lambda} > 1 $ we either have $ e-1 \leq e^{\lambda} - 1 $ or $ e^{\lambda} - 1 \leq \frac{1}{e} - 1 $.
		Thus we get
		\begin{equation*}
			\abs{e^{\lambda} - 1} \geq \min \setb{e-1, 1-\frac{1}{e}} \geq \frac{3}{5}.
		\end{equation*}
	\end{proof}
	
Before we state some lower bounds for linear forms in logarithms, let us recall the definition of the logarithmic height.
	Let $ \eta $ be an algebraic number of degree $ d $ over the rationals, with minimal polynomial
	\begin{equation*}
		a_0 (x - \eta_1) \cdots (x - \eta_d) \in \ZZ[x].
	\end{equation*}
	Then the absolute logarithmic height of $ \eta $ is given by
	\begin{equation*}
		h(\eta) := \frac{1}{d} \left( \log a_0 + \sum_{i=1}^{d} \log \max \setb{1, \abs{\eta_i}} \right).
	\end{equation*}
	This height function satisfies some basic properties, which are all well-known (see e.g.\ \cite{Zannier2009} for a reference). Namely for all $ \eta, \gamma \in \overline{\QQ} $ and all $ z \in \ZZ $ we have:
	\begin{align*}
		h(\eta + \gamma) &\leq h(\eta) + h(\gamma) + \log 2, \\
		h(\eta \gamma) &\leq h(\eta) + h(\gamma), \\
		h(\eta^z) &= \abs{z} h(\eta), \\
		h(z) &= \log |z| \qquad \text{for } z\neq 0.
	\end{align*}
	
The following lower bound for linear forms in logarithms is well-known and follows from Matveev's bound \cite{Matveev2000}.

\begin{myprop}[Matveev]\label{prop:Matveev}
Let $\eta_1, \ldots, \eta_t$ be positive real algebraic numbers in a number field $K$ of degree $D$, let $b_1, \ldots b_t$ be rational integers and assume that
\[
	\Lambda := b_1 \log \eta_1 + \dots + b_t \log \eta_t \neq 0.
\]
Then 
\[
	\log |\Lambda|
	\geq -1.4 \cdot 30^{t+3} \cdot t^{4.5} \cdot D^2 (1 + \log D) (1 + \log B) A_1 \cdots A_t,
\]
where
\begin{align*}
	B &\geq \max \setb{\abs{b_1}, \ldots, \abs{b_t}},\\
	A_i &\geq \max \setb{D h(\eta_i), \abs{\log \eta_i}, 0.16}
	\quad \text{for all }i = 1,\ldots,t.
\end{align*}
Moreover, the assumption $\Lambda\neq 0$ is equivalent to $e^\Lambda-1 \neq 0$ and the following bound holds as well:
\begin{align*}
	\log |e^\Lambda-1|
	&= \log |\eta_1^{b_1}\cdots \eta_t^{b_t} -1 |\\
	&\geq -1.4 \cdot 30^{t+3} \cdot t^{4.5} \cdot D^2 (1 + \log D) (1 + \log B) A_1 \cdots A_t.
\end{align*}
\end{myprop}
\begin{proof}
The bound for $|\Lambda|$ follows immediately from \cite[Corollary 2.3]{Matveev2000}. In fact, we have that
\begin{equation}\label{eq:Matveev-proof}
	\log |\Lambda|
	\geq - \frac{e}{2} \cdot 30^{t+3} \cdot t^{4.5} \cdot D^2 (1 + \log D) (1 + \log B) A_1 \cdots A_t,
\end{equation}
and we can estimate $e/2$ by $1.4$. 

The bound for $|e^\Lambda-1|$ now follows with Lemma~\ref{lemma:linsmall} and Lemma \ref{lemma:linbig}: If $|\Lambda|>1$, then $|e^\Lambda-1|\geq 3/5$, so $\log |e^\Lambda-1|\geq -0.52$ and the bound is trivially fulfilled. If  $|\Lambda|\leq 1$, then we have $|e^\Lambda-1|\geq \frac{1}{4} |\Lambda|$ and we obtain
\begin{align*}
	\log |e^\Lambda-1|
	\geq \log |\Lambda| - \log 4.
\end{align*}
Here the bound for $|e^\Lambda-1|$ follows from \eqref{eq:Matveev-proof} since we can omit the $\log 4$ when estimating $e/2$ by $1.4$.
\end{proof}

\begin{myrem}
The lower bound for linear forms in logarithms due to Baker and Wüstholz \cite{bawu93} played a significant role in the development of linear forms in logarithms. 
The final structure of the lower bound for linear forms in logarithms without an explicit determination of the constant involved has been established by Wüstholz \cite{wu88}, and the precise determination of that constant is the central aspect of \cite{bawu93}. 
However, slightly sharper bounds are obtained by using Matveev's result \cite{Matveev2000} instead. 

Let us note that using the highly technical result due to Mignotte \cite{Mignotte:kit} for linear forms in three logarithms would yield even smaller bounds, but to avoid technical difficulties we refrain from applying this result.
\end{myrem}

For linear forms in only two logarithms, there exist results with much smaller constants. In Section \ref{sec:Tribos} we will use the following bound, which follows immediately from  \cite{Laurent2008}.

\begin{myprop}[Laurent]\label{prop:Laurent}
Let $\eta_1, \eta_2 > 0$ be two real multiplicatively independent algebraic numbers, let $b_1, b_2 \in \ZZ$ be nonzero integers and let 
\[
	\Lambda := b_1 \log \eta_1 + b_2 \log \eta_2.
\]
Then
\[
	\log |\Lambda|
	\geq - 20.3 D^2 (\max( \log b' + 0.38, 18/D, 1))^2 \log A_1 \log A_2, 
\]
where
\begin{align*}
	D&=[\QQ(\eta_1,\eta_2):\QQ],\\
	b'&=\frac{|b_1|}{\log A_2} + \frac{|b_2|}{\log A_1},\\
	\log A_i &\geq \max ( D h(\eta_i),\abs{\log \eta_i},1)
		\qquad \text{for }i=1,2.
\end{align*}
\end{myprop}
\begin{proof}
The bound follows immediately from the bound \cite[Corollary 2]{Laurent2008} with $m=18$ (we have only hidden two of the $D$'s inside $\log A_1$ and $\log A_2$).

As for the assumptions, note that Laurent additionally supposes that $b_1$ is positive and $b_2$ is negative, and that $\log \eta_1$ and $\log \eta_2$ are positive. However, changing the signs of both $b$'s does not change $|\Lambda|$ and changing the sign of only one $b$ makes $\abs{\Lambda}$ larger. Thus it is enough to assume that $b_1,b_2$ are nonzero.
If $\log \eta_i$ is negative, one can simply swap $\eta_i$ for $\eta_i^{-1}$ and $b_i$ for $-b_i$ without changing the value of $\Lambda$ and still apply Laurent's result. Thus we can also allow the $\log \eta$'s to be negative. Moreover, we do not need to explicitly exclude $\log \eta_i=0$ as in that case we would have $\eta_i=1$, so the $\eta$'s would be multiplicatively dependent.
\end{proof}

Finally, at the end of Section \ref{sec:Tribos} we will apply the LLL-algorithm to reduce some of the obtained bounds. 
	The next lemma is an immediate variation of \cite[Lemma~VI.1]{Smart1998}.

\begin{mylemma}[LLL reduction]\label{lem:LLL}
Let $\gamma_1,\gamma_2,\gamma_3$ be positive real numbers and $x_1,x_2,x_3$ integers, and let
\[
	0 \neq |\Lambda|
	:= |x_1 \log \gamma_1 + x_2 \log \gamma_2 + x_3 \log \gamma_3|.
\]
Assume that the $x_i$ are bounded in absolute values by some constant $M$ and choose a constant $C>M^3$. 
Consider the matrix
\[
	A= \begin{pmatrix}
	1 & 0 & 0 \\
	0 & 1 & 0 \\
	[C \log \gamma_1] & [C \log \gamma_2] &[C \log \gamma_3]
	\end{pmatrix},
\]
where $[x]$ denotes the nearest integer to $x$. The columns of $A$ form a basis of a lattice. Let $B$ be the matrix that corresponds to the LLL-reduced basis and let $B^*$ be the matrix corresponding to the Gram-Schmidt basis constructed from $B$. Let $c$ be the Euclidean norm of the smallest column vector of $B^*$ and set $S:=2M^2$ and $T:=(1+3M)/2$.
If $c^2 > T^2 + S$, then 
\[
	|\Lambda|
	> \frac{1}{C}\left( \sqrt{c^2-S} - T \right).
\]
\end{mylemma}

\section{Overview of the proofs}\label{sec:overviewProof}

The main idea in the proofs of Theorem~\ref{thm:mainthm} and Theorem~\ref{thm:Tribos} is to consider equations of the form $U_{n_1}-b^{m_1}=U_{n_2}-b^{m_2}$ and note that they are very roughly of the form $a \alpha^{n_1} - b^{m_1}= a \alpha^{n_2} - b^{m_2}$. Then one can do the usual tricks for solving Pillai-type equations via lower bounds for linear forms in logarithms: shifting expressions, estimating and applying e.g.\ Matveev's lower bound. This needs to be done in several steps. The main problem in our situation is that $b$ is not fixed, which is why we need to eliminate it at some point and why we assume the existence of three solutions.

We split up the proof of Theorem~\ref{thm:mainthm} into several lemmas and the proof of Theorem~\ref{thm:Tribos} into several steps.
Table \ref{table:overview} shows a simplified overview of the main lemmas/steps (the cases which allow us to skip some steps are left out in the table). In each lemma/step a new bound is obtained. Note that the symbol $\ll$ stands for ``$\leq$ up to some effectively computable constant''. Lemmas and steps which use the same linear form in logarithms are displayed side by side. Steps marked with a *star do not involve linear forms in logarithms but simply combine previous bounds.

\begin{table}[h]
\caption{Overview of the proofs}\label{table:overview}
\begin{tabular}{l|l}
\hline
\textbf{Proof of Theorem \ref{thm:mainthm}} 					& \textbf{Proof of Theorem \ref{thm:Tribos}} \\ 
(general result) 												& (Tribonacci numbers) \\ 
\hline
\textbf{Lemma \ref{lemma:step1}:}           					& \textbf{Step \ref{step:Step1}:}                                            
\\
$\min \setb{n_1-n_2, (m_1-m_2) \log b} \ll \log n_1 \log b$  	&  $n_1-n_2 \ll \log n_1 \log b$
\\
\textbf{Lemma \ref{lemma:step2}:}           					&                              
\\
$\max \setb{n_1-n_2, (m_1-m_2) \log b} \ll (\log n_1 \log b)^2$ &                                                                                                         \\
\textbf{Lemma \ref{lemma:step3}:}           					& \textbf{Step \ref{step:Step2}:}                                             \\
$n_1 \ll (\log n_1 \log b)^3$                                	& $n_1 \ll (\log n_1 \log b)^2$                                                                           \\ \cline{1-1}
\textbf{Lemma \ref{lemma:step5}:}                             	& \textbf{Step \ref{step:Step3}:}                                             \\
$n_2-n_3 \ll \log n_1$           								& $n_2-n_3 \ll (\log n_1)^2$                                                                             
\\
\textbf{Lemma \ref{lemma:step6}:} 								& \textbf{Step \ref{step:Step4}:}                                             \\
$\log b \ll (\log n_1)^2$                                     	& $n_1-n_2 \ll (\log n_1)^2$                                                                              \\
									           					& \textbf{*Step \ref{step:Step5}:}                                             \\
							                                   	& $\log b \ll (\log n_1)^2$                                                                               \\
\textbf{*End:}                                                	& \textbf{*Step \ref{step:Step6}:}                                             \\
$n_1 \ll (\log n_1)^9$                                       	& $n_1 \ll (\log n_1)^8$                                                                                  \\
                                                             	& \textbf{Reduction Steps}                                                                                \\ \hline
\end{tabular}
\end{table}

In the proof of Theorem~\ref{thm:mainthm} in the first three lemmas we only assume the existence of two solutions. We do that because we want to show how far we can get with our method if we aim at proving a stronger result. Below the horizontal line in the middle of Table~\ref{table:overview} we assume the existence of three solutions. 
In the proof of Theorem~\ref{thm:Tribos} we assume the existence of three solutions from the beginning. This saves us one step and thus helps to keep the constants and exponents smaller (see Step~\ref{step:Step2} in the table). The attentive reader will notice that the bound in Step~\ref{step:Step3} has a larger exponent than the bound in Lemma ~\ref{lemma:step5}. This is because the corresponding linear form has only two logarithms, so in the Tribonacci setting we apply Laurent's lower bound instead of Matveev's, trading the exponent for a much better constant.

In the end, in both proofs we obtain an absolute upper bound for $n_1$ from an inequality of the form $n_1 \ll (\log n_1)^k$. Since the constant is very large, the bound for $n_1$ is also very large. In the Tribonacci setting we try to reduce it. Unfortunately, since $b$ is not fixed and the bound for $b$ is very large, we cannot use the linear forms which contain $b$ for the reduction. This is why we are not able to solve Problem~\ref{problem:Tribos-complete}. However, we can still use the linear forms in logarithms which do not contain $b$ and reduce the initial bound for $n_1$ significantly.
	
	\section{Proof of Theorem \ref{thm:mainthm}}\label{sec:proofMainThm}
	
	Since the proof is a little bit lengthy, we will split it into several lemmas.
	All assumptions listed in Theorem \ref{thm:mainthm} are general assumptions in this section.
	During the whole section we will denote by $ C_i $ effectively computable positive constants which are independent of $ n,m,b,c $.
	We may assume that $ b \geq 2 $ without loss of generality.
	
	Let us now fix a Galois automorphism $ \sigma $ 
	on the splitting field $\QQ(\alpha, \alpha_2 , \ldots , \alpha_k)$
	with the property $ \abs{\sigma(\alpha)} < \alpha $.
	This is always possible because $ \alpha $ is irrational and the dominant root of $ U_n $.
	The effectively computable constant $ N_0 $ will only depend on the characteristic roots and coefficients of the recurrence $ U_n $ as well as on the chosen automorphism $ \sigma $.
	Although it possibly will be updated at some points in the proof, we will always denote it by the same label $ N_0 $.
	As a first step we will choose $ N_0 $ large enough such that $ U_n $ is positive and strictly increasing for all $ n \geq N_0 $.
	This is possible since $ U_n $ has a dominant root $ \alpha > 1 $ with positive coefficient.
	In addition we can assume that $ \abs{\alpha_2} \geq \abs{\alpha_i} $ for $ i = 2,\ldots,k $ without loss of generality.
	
	Assuming that Equation \eqref{eq:centraleq} has three distinct solutions $ (n_1,m_1) $, $ (n_2,m_2) $ and $ (n_3,m_3) $, we can write
	\begin{equation*}
		U_{n_1} - b^{m_1} = U_{n_2} - b^{m_2} = U_{n_3} - b^{m_3}
	\end{equation*}
	with $ n_1 > n_2 > n_3 \geq N_0 $ and $ m_1 > m_2 > m_3 \geq 1 $.
	When working only with two of those solutions we often may write
	\begin{equation}
		\label{eq:twosols}
		U_{n_1} - U_{n_2} = b^{m_1} - b^{m_2}.
	\end{equation}
	Let us fix
	\begin{equation*}
		\varepsilon := \frac{1}{2} \cdot \frac{\alpha-1}{\alpha+1} > 0.
	\end{equation*}
	Recalling that $ \alpha $ is the dominant root of $ U_n $, we get
	\begin{equation*}
		\abs{\frac{U_n}{a \alpha^n} - 1} \leq \varepsilon
	\end{equation*}
	for $ n \geq N_0 $ (where we might have updated $N_0$). Thus for $ n \geq N_0 $ we have
	\begin{equation*}
		\abs{U_n - a \alpha^n} \leq \varepsilon a \alpha^n,
	\end{equation*}
	which implies the bounds
	\begin{equation}
		\label{eq:unrange}
		C_1 a \alpha^n \leq U_n \leq C_2 a \alpha^n
	\end{equation}
	for $ C_1 = 1 - \varepsilon $ and $ C_2 = 1 + \varepsilon $.
	The choice of $ \varepsilon $ guarantees that $ C_3 := C_1 \alpha - C_2 > 0 $.
	Therefore by using Equation \eqref{eq:twosols} and Inequality \eqref{eq:unrange} we get
	\begin{align*}
		b^{m_1} &\geq b^{m_1} - b^{m_2} = U_{n_1} - U_{n_2} \geq C_1 a \alpha^{n_1} - C_2 a \alpha^{n_2} \\
		&= C_1 \alpha a \alpha^{n_1-1} - C_2 a \alpha^{n_2} \geq C_3 a \alpha^{n_1-1} = C_4 \alpha^{n_1}
	\end{align*}
	as well as
	\begin{equation*}
		C_2 a \alpha^{n_1} \geq U_{n_1} \geq U_{n_1} - U_{n_2} = b^{m_1} - b^{m_2} \geq \left( 1 - \frac{1}{b} \right) b^{m_1} \geq \frac{1}{2} b^{m_1}.
	\end{equation*}
	These two inequality chains yield
	\begin{equation}
		\label{eq:relleadterm}
		C_4 \alpha^{n_1} \leq b^{m_1} \leq C_5 \alpha^{n_1}.
	\end{equation}
	Note that these bounds are valid for any solution of \eqref{eq:centraleq} provided that there is a further smaller solution.
	Applying the logarithm to Equation \eqref{eq:relleadterm} gives us
	\begin{equation*}
		m_1 \leq \frac{\log C_5}{\log b} + n_1 \cdot \frac{\log \alpha}{\log b} < n_1
	\end{equation*}
	if $ b $ is large enough.
	Thus for our purpose we may assume that $ n_1 > m_1 $.
	
	The next big intermediate result is to prove that, if there exist at least two distinct solutions $ (n_1,m_1) $ and $ (n_2,m_2) $ to Equation \eqref{eq:centraleq}, then for the larger one $ (n_1,m_1) $ the bound
	\begin{equation}
		\label{eq:logbound}
		n_1 \leq C_6 (\log n_1 \log b)^3
	\end{equation}
	holds.
	This will be done by the following three lemmas.
	
	\begin{mylemma}
		\label{lemma:step1}
		Assume that Equation \eqref{eq:centraleq} has at least two distinct solutions $ (n_1,m_1) $ and $ (n_2,m_2) $ with $ n_1 > n_2 $ as considered in Theorem \ref{thm:mainthm}.
		Then we have
		\begin{equation*}
			\min \setb{n_1-n_2, (m_1-m_2) \log b} \leq C_7 \log n_1 \log b.
		\end{equation*}
	\end{mylemma}
	
	\begin{proof}
		Inserting the Binet representation of the linear recurrence sequence into Equation \eqref{eq:twosols} and regrouping terms yields
		\begin{align*}
			\abs{a \alpha^{n_1} - b^{m_1}} &= \abs{a \alpha^{n_2} + \sum_{l=2}^{k} a_l \alpha_l^{n_2} - \sum_{l=2}^{k} a_l \alpha_l^{n_1} - b^{m_2}} \\
			&\leq a \alpha^{n_2} + C_8 \abs{\alpha_2}^{n_2} + C_9 \abs{\alpha_2}^{n_1} + b^{m_2},
		\end{align*}
		and, dividing both sides by $ b^{m_1} $, we get
		\begin{align}
			\abs{\frac{a \alpha^{n_1}}{b^{m_1}} - 1} &\leq \frac{1}{b^{m_1}} \cdot \left( a \alpha^{n_2} + C_8 \abs{\alpha_2}^{n_2} + C_9 \abs{\alpha_2}^{n_1} + b^{m_2} \right) \nonumber \\
			&= \frac{a \alpha^{n_2}}{b^{m_1}} + C_8 \frac{\abs{\alpha_2}^{n_2}}{b^{m_1}} + C_9 \frac{\abs{\alpha_2}^{n_1}}{b^{m_1}} + b^{m_2 - m_1} \nonumber \\
			\label{eq:chain_step1}
			&\leq \frac{a \alpha^{n_2}}{C_4 \alpha^{n_1}} + C_8 \frac{\abs{\alpha_2}^{n_2}}{C_4 \alpha^{n_1}} + C_9 \frac{\abs{\alpha_2}^{n_1}}{C_4 \alpha^{n_1}} + b^{m_2 - m_1} \\
			&\leq \frac{a}{C_4} \alpha^{n_2-n_1} + \frac{C_8}{C_4} \alpha^{n_2-n_1} + \frac{C_9}{C_4} \abs{\alpha_2}^{n_1-n_2} \alpha^{n_2-n_1} + b^{m_2 - m_1} \nonumber \\
			&\leq C_{10} \max \setb{\alpha^{n_2-n_1}, \left( \frac{\alpha}{\abs{\alpha_2}} \right)^{n_2-n_1}, b^{m_2 - m_1}},\nonumber 
		\end{align}
		where in the third line we have used Inequality \eqref{eq:relleadterm}.
		Note that we can either have $ \abs{\alpha_2} \leq 1 $ or $ \abs{\alpha_2} > 1 $ and that the bound may be simplified if we work with a concrete recurrence sequence where the size of $ \alpha_2 $ is known.
		
		We aim for applying Proposition~\ref{prop:Matveev} to get also a lower bound.
		Therefore let us set $ t=3 $, $ K = \QQ(a,\alpha) $, $ D = [K:\QQ] $ as well as 
	\begin{align*}
		\eta_1 &= a, \quad  & 
		\eta_2 &= \alpha, \quad &
		\eta_3 &= b,\\
		b_1 &= 1, \quad &
		b_2 &= n_1, \quad &
		b_3 &= -m_1.
	\end{align*}
		Further we can put
		\begin{align*}
			A_1 &= \max \setb{Dh(a), \abs{\log a}, 0.16}, \\
			A_2 &= \max \setb{Dh(\alpha), \log \alpha, 0.16}, \\
			A_3 &= D \log b, \\
			B &= n_1.
		\end{align*}
		Finally, we have to show that the expression
		\begin{equation*}
			\Lambda := a \alpha^{n_1} b^{-m_1} - 1
		\end{equation*}
		is nonzero.
		Assuming the contrary, we would have $ a \alpha^{n_1} = b^{m_1} $ 		
		and by applying $ \sigma $ we would get the equality $ a \alpha^{n_1} = \sigma(a) \sigma(\alpha)^{n_1} $.
		Taking absolute values this implies
		\begin{equation*}
			\frac{\abs{\sigma(a)}}{a} = \left( \frac{\alpha}{\abs{\sigma(\alpha)}} \right)^{n_1},
		\end{equation*}
		which is a contradiction for $ n_1 \geq N_0 $ (where, as always, we may have updated $N_0$). 
		Hence we have $ \Lambda \neq 0 $.
		Now Proposition~\ref{prop:Matveev} states that
		\begin{equation*}
			\log \abs{\Lambda} \geq -C_{11} (1 + \log n_1) \log b \geq -C_{12} \log n_1 \log b.
		\end{equation*}
		
		Comparing this lower bound with the upper bound coming from the first paragraph of this proof gives us
		\begin{equation*}
			-C_{12} \log n_1 \log b \leq \log C_{10} + \log \max \setb{\alpha^{n_2-n_1}, \left( \frac{\alpha}{\abs{\alpha_2}} \right)^{n_2-n_1}, b^{m_2 - m_1}}
		\end{equation*}
		which immediately implies
		\begin{equation*}
			\min \setb{n_1-n_2, (m_1-m_2) \log b} \leq C_7 \log n_1 \log b.
		\end{equation*}
		This proves the lemma.
	\end{proof}
	
	\begin{mylemma}
		\label{lemma:step2}
		Assume that Equation \eqref{eq:centraleq} has at least two distinct solutions $ (n_1,m_1) $ and $ (n_2,m_2) $ with $ n_1 > n_2 $ as considered in Theorem \ref{thm:mainthm}.
		Then we have
		\begin{equation*}
			\max \setb{n_1-n_2, (m_1-m_2) \log b} \leq C_{13} (\log n_1 \log b)^2
		\end{equation*}
		or
		\begin{equation*}
			n_1 \leq C_{14} (\log n_1 \log b)^2.
		\end{equation*}
	\end{mylemma}
	
	\begin{proof}
		We distinguish between two cases according to the statement of Lemma \ref{lemma:step1}.
		Let us first assume that
		\begin{equation*}
			\min \setb{n_1-n_2, (m_1-m_2) \log b} = n_1-n_2.
		\end{equation*}
		Then we have
		\begin{equation}
			\label{eq:s2case1}
			n_1-n_2 \leq C_7 \log n_1 \log b.
		\end{equation}
		Inserting the Binet representation of the linear recurrence sequence into Equation~\eqref{eq:twosols} and regrouping terms yields
		\begin{align*}
			\abs{a \alpha^{n_2} (\alpha^{n_1-n_2} - 1) - b^{m_1}} &= \abs{\sum_{l=2}^{k} a_l \alpha_l^{n_2} - \sum_{l=2}^{k} a_l \alpha_l^{n_1} - b^{m_2}} \\
			&\leq C_8 \abs{\alpha_2}^{n_2} + C_9 \abs{\alpha_2}^{n_1} + b^{m_2}.
		\end{align*}
		Now we divide this inequality by $ b^{m_1} $.
		If $ \abs{\alpha_2} \leq 1 $ we get
		\begin{equation*}
			\abs{\frac{a \alpha^{n_2} (\alpha^{n_1-n_2} - 1)}{b^{m_1}} - 1} \leq \frac{C_8 + C_9 + b^{m_2}}{b^{m_1}} \leq C_{15} b^{m_2-m_1},
		\end{equation*}
		and if $ \abs{\alpha_2} > 1 $ we get
		\begin{align*}
			\abs{\frac{a \alpha^{n_2} (\alpha^{n_1-n_2} - 1)}{b^{m_1}} - 1} &\leq \frac{C_{16} \abs{\alpha_2}^{n_1} + b^{m_2}}{b^{m_1}} \leq \frac{C_{16} \abs{\alpha_2}^{n_1}}{C_4 \alpha^{n_1}} + b^{m_2-m_1} \\
			&\leq C_{17} \max \setb{\left( \frac{\abs{\alpha_2}}{\alpha} \right)^{n_1}, b^{m_2-m_1}},
		\end{align*}
		where for the second inequality we have used \eqref{eq:relleadterm}.
		So in both subcases we have
		\begin{equation}
			\label{eq:s2c1upper}
			\abs{\frac{a \alpha^{n_2} (\alpha^{n_1-n_2} - 1)}{b^{m_1}} - 1} \leq C_{18} \max \setb{\left( \frac{\abs{\alpha_2}}{\alpha} \right)^{n_1}, b^{m_2-m_1}}.
		\end{equation}
		
		We aim for applying Proposition~\ref{prop:Matveev} to get a lower bound.
		Therefore let us set $ t=3 $, $ K = \QQ(a,\alpha) $, $ D = [K:\QQ] $ as well as 
	\begin{align*}
		\eta_1 &= a(\alpha^{n_1-n_2} - 1), \quad &
		\eta_2 &= \alpha, \quad &
		\eta_3 &= b, \\
		b_1 &= 1, \quad &
		b_2 &= n_2, \quad &
		b_3 &=-m_1.
	\end{align*}			
		Using Inequality \eqref{eq:s2case1} gives us
		\begin{align*}
			h(\eta_1) &\leq h(a) + h(\alpha^{n_1-n_2} - 1) \leq h(a) + h(\alpha^{n_1-n_2}) + h(1) + \log 2 \\
			&\leq C_{19} + (n_1-n_2) h(\alpha) \leq C_{20} \log n_1 \log b
		\end{align*}
		and an analogous bound for $ \log \eta_1 $.
		Thus we can put
		\begin{align*}
			A_1 &= C_{21} \log n_1 \log b, \\
			A_2 &= \max \setb{Dh(\alpha), \log \alpha, 0.16}, \\
			A_3 &= D \log b, \\
			B &= n_1.
		\end{align*}
		Finally, we have to show that the expression
		\begin{equation*}
			\Lambda := a(\alpha^{n_1-n_2} - 1) \alpha^{n_2} b^{-m_1} - 1
		\end{equation*}
		is nonzero.
		Assuming the contrary, we would have $ a (\alpha^{n_1-n_2} - 1) \alpha^{n_2} = b^{m_1} $ and by applying $ \sigma $ the equality $ a (\alpha^{n_1-n_2} - 1) \alpha^{n_2} = \sigma(a) (\sigma(\alpha)^{n_1-n_2} - 1) \sigma(\alpha)^{n_2} $.
		Taking absolute values this implies
		\begin{equation*}
			\abs{\frac{\sigma(a) (\sigma(\alpha)^{n_1-n_2} - 1)}{a (\alpha^{n_1-n_2} - 1)}} = \left( \frac{\alpha}{\abs{\sigma(\alpha)}} \right)^{n_2},
		\end{equation*}
		which is a contradiction for $ n_2 \geq N_0 $ since the left hand side is bounded above by a constant.
		Hence we have $ \Lambda \neq 0 $.
		Now Proposition~\ref{prop:Matveev} states that
		\begin{equation*}
			\log \abs{\Lambda} \geq -C_{22} (1 + \log n_1) \log n_1 \log b \log b \geq -C_{23} (\log n_1 \log b)^2.
		\end{equation*}
		
		Comparing this lower bound with the upper bound \eqref{eq:s2c1upper} gives us
		\begin{equation*}
			-C_{23} (\log n_1 \log b)^2 \leq \log C_{18} + \log \max \setb{\left( \frac{\abs{\alpha_2}}{\alpha} \right)^{n_1}, b^{m_2-m_1}}
		\end{equation*}
		which immediately implies
		\begin{equation*}
			\min \setb{n_1, (m_1-m_2) \log b} \leq C_{24} (\log n_1 \log b)^2.
		\end{equation*}
		Then, depending on which of the two expressions is the minimum, we either have
		\begin{equation*}
			n_1 \leq C_{14} (\log n_1 \log b)^2
		\end{equation*}
		or
		\begin{equation*}
			\max \setb{n_1-n_2, (m_1-m_2) \log b} = (m_1-m_2) \log b \leq C_{13} (\log n_1 \log b)^2.
		\end{equation*}
		This concludes the first case.
		
		In the second case we assume that
		\begin{equation*}
			\min \setb{n_1-n_2, (m_1-m_2) \log b} = (m_1-m_2) \log b.
		\end{equation*}
		Thus we have
		\begin{equation}
			\label{eq:s2case2}
			(m_1-m_2) \log b \leq C_7 \log n_1 \log b.
		\end{equation}
		Inserting the Binet representation of the linear recurrence sequence into Equation~\eqref{eq:twosols} and regrouping terms yields
		\begin{align*}
			\abs{a \alpha^{n_1} - b^{m_2} (b^{m_1-m_2} - 1)} &= \abs{a \alpha^{n_2} + \sum_{l=2}^{k} a_l \alpha_l^{n_2} - \sum_{l=2}^{k} a_l \alpha_l^{n_1}} \\
			&\leq a \alpha^{n_2} + C_8 \abs{\alpha_2}^{n_2} + C_9 \abs{\alpha_2}^{n_1} \\
			&\leq C_{25} \alpha^{n_2} + C_9 \abs{\alpha_2}^{n_1},
		\end{align*}
		and, dividing both sides by $ b^{m_1} - b^{m_2} $, we get
		\begin{align}
			\abs{\frac{a \alpha^{n_1}}{b^{m_2} (b^{m_1-m_2} - 1)} - 1} &\leq \frac{C_{25} \alpha^{n_2} + C_9 \abs{\alpha_2}^{n_1}}{b^{m_1} - b^{m_2}} \nonumber \\
			&\leq \frac{C_{25} \alpha^{n_2} + C_9 \abs{\alpha_2}^{n_1}}{\frac{1}{2} b^{m_1}} \nonumber \\
			&\leq \frac{C_{25} \alpha^{n_2} + C_9 \abs{\alpha_2}^{n_1}}{\frac{1}{2} C_4 \alpha^{n_1}} \nonumber \\
			&\leq C_{26} \alpha^{n_2-n_1} + C_{27} \abs{\alpha_2}^{n_1-n_2} \alpha^{n_2-n_1} \nonumber \\
			\label{eq:s2c2upper}
			&\leq C_{28} \max \setb{\alpha^{n_2-n_1}, \left( \frac{\alpha}{\abs{\alpha_2}} \right)^{n_2-n_1}},
		\end{align}
		where we have used the inequalities $ b^{m_1} - b^{m_2} \geq \frac{1}{2} b^{m_1} $ and \eqref{eq:relleadterm}.
		
		We aim for applying Proposition~\ref{prop:Matveev} to get a lower bound.
		Therefore let us set $ t=3 $, $ K = \QQ(a,\alpha) $, $ D = [K:\QQ] $ as well as
\begin{align*}
\eta_1 &= (b^{m_1-m_2} - 1)/a, \quad &
\eta_2 &= \alpha, \quad &
\eta_3 &= b, \\
b_1 &= -1, \quad &
b_2 &= n_1, \quad &
b_3 &= -m_2.
\end{align*}		
		Using Inequality \eqref{eq:s2case2} as well as $ h(b) = \log b $ gives us
		\begin{align*}
			h(\eta_1) &\leq h(a) + h(b^{m_1-m_2} - 1) \leq h(a) + h(b^{m_1-m_2}) + h(1) + \log 2 \\
			&\leq C_{29} + (m_1-m_2) h(b) \leq C_{30} \log n_1 \log b
		\end{align*}
		and an analogous bound for $ \log \eta_1 $.
		Thus we can put
		\begin{align*}
			A_1 &= C_{31} \log n_1 \log b, \\
			A_2 &= \max \setb{Dh(\alpha), \log \alpha, 0.16}, \\
			A_3 &= D \log b, \\
			B &= n_1.
		\end{align*}
		Finally, we have to show that the expression
		\begin{equation*}
			\Lambda := \alpha^{n_1} b^{-m_2} ((b^{m_1-m_2} - 1)/a)^{-1} - 1
		\end{equation*}
		is nonzero.
		Assuming the contrary, we would have $ a \alpha^{n_1} = b^{m_1} - b^{m_2} $ and by applying $ \sigma $ we would get the equality $ a \alpha^{n_1} = \sigma(a) \sigma(\alpha)^{n_1} $.
		Taking absolute values this implies
		\begin{equation*}
			\frac{\abs{\sigma(a)}}{a} = \left( \frac{\alpha}{\abs{\sigma(\alpha)}} \right)^{n_1},
		\end{equation*}
		which is a contradiction for $ n_1 \geq N_0 $.
		Hence we have $ \Lambda \neq 0 $.
		Now Proposition~\ref{prop:Matveev} states that
		\begin{equation*}
			\log \abs{\Lambda} \geq -C_{32} (1 + \log n_1) \log n_1 \log b \log b \geq -C_{33} (\log n_1 \log b)^2.
		\end{equation*}
		
		Comparing this lower bound with the upper bound \eqref{eq:s2c2upper} gives us
		\begin{equation*}
			-C_{33} (\log n_1 \log b)^2 \leq \log C_{28} + \log \max \setb{\alpha^{n_2-n_1}, \left( \frac{\alpha}{\abs{\alpha_2}} \right)^{n_2-n_1}}
		\end{equation*}
		which immediately implies
		\begin{equation*}
			\max \setb{n_1-n_2, (m_1-m_2) \log b} = n_1-n_2 \leq C_{13} (\log n_1 \log b)^2.
		\end{equation*}
		This concludes the proof of the lemma.
	\end{proof}
	
	\begin{mylemma}
		\label{lemma:step3}
		Assume that Equation \eqref{eq:centraleq} has at least two distinct solutions $ (n_1,m_1) $ and $ (n_2,m_2) $ with $ n_1 > n_2 $ as considered in Theorem \ref{thm:mainthm}.
		Then we have
		\begin{equation*}
			n_1 \leq C_6 (\log n_1 \log b)^3.
		\end{equation*}
	\end{mylemma}
	
	\begin{proof}
		By Lemma \ref{lemma:step2} we can assume that
		\begin{equation}
			\label{eq:s2bound}
			\max \setb{n_1-n_2, (m_1-m_2) \log b} \leq C_{13} (\log n_1 \log b)^2
		\end{equation}
		since the other case is trivial.
		Inserting the Binet representation of the linear recurrence sequence into Equation \eqref{eq:twosols} and regrouping terms yields
		\begin{align*}
			\abs{a \alpha^{n_2} (\alpha^{n_1-n_2} - 1) - b^{m_2} (b^{m_1-m_2} - 1)} &= \abs{\sum_{l=2}^{k} a_l \alpha_l^{n_2} - \sum_{l=2}^{k} a_l \alpha_l^{n_1}} \\
			&\leq C_8 \abs{\alpha_2}^{n_2} + C_9 \abs{\alpha_2}^{n_1}.
		\end{align*}
		Dividing both sides by $ b^{m_1} - b^{m_2} $ then gives us
		\begin{align*}
			\abs{\frac{a \alpha^{n_2} (\alpha^{n_1-n_2} - 1)}{b^{m_2} (b^{m_1-m_2} - 1)} - 1} &\leq \frac{C_8 \abs{\alpha_2}^{n_2} + C_9 \abs{\alpha_2}^{n_1}}{b^{m_1} - b^{m_2}} \\
			&\leq \frac{C_8 \abs{\alpha_2}^{n_2} + C_9 \abs{\alpha_2}^{n_1}}{\frac{1}{2} b^{m_1}} \\
			&\leq \frac{C_8 \abs{\alpha_2}^{n_2} + C_9 \abs{\alpha_2}^{n_1}}{\frac{1}{2} C_4 \alpha^{n_1}} \\
			&\leq C_{34} \max \setb{\alpha^{-n_1}, \left( \frac{\alpha}{\abs{\alpha_2}} \right)^{-n_1}},
		\end{align*}
		where we have used the inequalities $ b^{m_1} - b^{m_2} \geq \frac{1}{2} b^{m_1} $ and \eqref{eq:relleadterm}.
		Note that the maximum occurs in view of the distinction between $ \abs{\alpha_2} \leq 1 $ and $ \abs{\alpha_2} > 1 $.
		
		We aim for applying Proposition~\ref{prop:Matveev} to get a lower bound.
		Therefore let us set $ t=3 $, $ K = \QQ(a,\alpha) $, $ D = [K:\QQ] $ as well as 
\begin{align*}
\eta_1 &= \frac{a(\alpha^{n_1-n_2} - 1)}{b^{m_1-m_2} - 1}, \quad &
\eta_2 &= \alpha, \quad &
\eta_3 &= b,\\
b_1 &= 1, \quad &
b_2 &= n_2, \quad &
b_3 &= -m_2.
\end{align*}		
		Using the bound \eqref{eq:s2bound} as well as $ h(b) = \log b $ gives us
		\begin{align*}
			h(\eta_1) &\leq h(a) + h(\alpha^{n_1-n_2} - 1) + h(b^{m_1-m_2} - 1) \\
			&\leq h(a) + h(\alpha^{n_1-n_2}) + h(1) + \log 2 + h(b^{m_1-m_2}) + h(1) + \log 2 \\
			&\leq C_{35} + (n_1-n_2) h(\alpha) + (m_1-m_2) h(b) \leq C_{36} (\log n_1 \log b)^2
		\end{align*}
		and an analogous bound for $ \abs{\log \eta_1} $.
		Thus we can put
		\begin{align*}
			A_1 &= C_{37} (\log n_1 \log b)^2, \\
			A_2 &= \max \setb{Dh(\alpha), \log \alpha, 0.16}, \\
			A_3 &= D \log b, \\
			B &= n_1.
		\end{align*}
		Finally, we have to show that the expression
		\begin{equation*}
			\Lambda := \alpha^{n_2} b^{-m_2} a(\alpha^{n_1-n_2} - 1)/(b^{m_1-m_2} - 1) - 1
		\end{equation*}
		is nonzero.
		Assuming the contrary, we would have $ a (\alpha^{n_1-n_2} - 1) \alpha^{n_2} = b^{m_1} - b^{m_2} $ and by applying $ \sigma $ we get the equality $ a (\alpha^{n_1-n_2} - 1) \alpha^{n_2} = \sigma(a) (\sigma(\alpha)^{n_1-n_2} - 1) \sigma(\alpha)^{n_2} $.
		Taking absolute values this implies
		\begin{equation*}
			\abs{\frac{\sigma(a) (\sigma(\alpha)^{n_1-n_2} - 1)}{a (\alpha^{n_1-n_2} - 1)}} = \left( \frac{\alpha}{\abs{\sigma(\alpha)}} \right)^{n_2},
		\end{equation*}
		which is a contradiction for $ n_2 \geq N_0 $ since the left hand side is bounded above by a constant.
		Hence we have $ \Lambda \neq 0 $.
		Now Proposition~\ref{prop:Matveev} states that
		\begin{equation*}
			\log \abs{\Lambda} \geq -C_{38} (1 + \log n_1) (\log n_1 \log b)^2 \log b \geq -C_{39} (\log n_1 \log b)^3.
		\end{equation*}
		
		Comparing this lower bound with the upper bound coming from the first paragraph of this proof gives us
		\begin{equation*}
			-C_{39} (\log n_1 \log b)^3 \leq \log C_{34} + \log \max \setb{\alpha^{-n_1}, \left( \frac{\alpha}{\abs{\alpha_2}} \right)^{-n_1}}
		\end{equation*}
		which immediately implies
		\begin{equation*}
			n_1 \leq C_6 (\log n_1 \log b)^3.
		\end{equation*}
		Thus the lemma is proven.
	\end{proof}
	
	We have now reached our first milestone, the bound \eqref{eq:logbound} is proven.
	The second milestone is to prove that if there are at least three distinct solutions $ (n_1,m_1) $, $ (n_2,m_2) $ and $ (n_3,m_3) $, then for the largest one we have
	\begin{equation}
		\label{eq:bbound}
		\log b \leq C_{40} (\log n_1)^2.
	\end{equation}
	Again we will split this part into some lemmas.
	
	\begin{mylemma}
		\label{lemma:step4}
		Assume that Equation \eqref{eq:centraleq} has at least two distinct solutions $ (n_1,m_1) $ and $ (n_2,m_2) $ with $ n_1 > n_2 $ as considered in Theorem \ref{thm:mainthm}.
		Then we have
		\begin{equation*}
			n_1 \log \alpha - C_{41} \leq m_1 \log b \leq n_1 \log \alpha + C_{42}.
		\end{equation*}
	\end{mylemma}
	
	\begin{proof}
		This follows immediately by applying the logarithm to Inequality \eqref{eq:relleadterm} from above.
	\end{proof}
	
	\begin{mylemma}
		\label{lemma:step5}
		Assume that Equation \eqref{eq:centraleq} has at least three distinct solutions $ (n_1,m_1) $, $ (n_2,m_2) $ and $ (n_3,m_3) $ with $ n_1 > n_2 > n_3 $ as considered in Theorem \ref{thm:mainthm}.
		Then at least one of the following inequalities holds:
		\begin{enumerate}[(i)]
			\item $ \log b \leq C_{43} \log n_1 $,
			\item $ n_2-n_3 \leq C_{44} \log n_1 $.
		\end{enumerate}
	\end{mylemma}
	
	\begin{proof}
		Let us recall Inequality \eqref{eq:chain_step1} from the proof of Lemma \ref{lemma:step1}, where we obtained
		\begin{equation*}
			\abs{\frac{a \alpha^{n_1}}{b^{m_1}} - 1} \leq C_{10} \max \setb{\alpha^{n_2-n_1}, \left( \frac{\alpha}{\abs{\alpha_2}} \right)^{n_2-n_1}, b^{m_2 - m_1}}.
		\end{equation*}
		Now we define the linear form
		\begin{equation*}
			\Lambda_{12} := n_1 \log \alpha - m_1 \log b + \log a = \log \frac{a \alpha^{n_1}}{b^{m_1}}
		\end{equation*}
		and distinguish between two cases.
		Let us first assume that $ \abs{\Lambda_{12}} > 1 $.
		Then by Lemma \ref{lemma:linbig} we have
		\begin{equation*}
			\frac{3}{5} \leq \abs{e^{\Lambda_{12}} - 1} \leq C_{10} \max \setb{\alpha^{n_2-n_1}, \left( \frac{\alpha}{\abs{\alpha_2}} \right)^{n_2-n_1}, b^{m_2 - m_1}}.
		\end{equation*}
		If the maximum is either $ \alpha^{n_2-n_1} $ or $ \left( \frac{\alpha}{\abs{\alpha_2}} \right)^{n_2-n_1} $, this implies an upper bound $ n_1-n_2 \leq C_{45} $.
		We will come back to this later.
		If the maximum is $ b^{m_2 - m_1} $, then the inequality implies
		\begin{equation*}
			\log b \leq (m_1-m_2) \log b \leq C_{46}
		\end{equation*}
		and we are done.
		Therefore we will now assume that $ \abs{\Lambda_{12}} \leq 1 $.
		
		Since we are now working with three solutions, we analogously get the upper bound
		\begin{equation*}
			\abs{\frac{a \alpha^{n_2}}{b^{m_2}} - 1} \leq C_{10} \max \setb{\alpha^{n_3-n_2}, \left( \frac{\alpha}{\abs{\alpha_2}} \right)^{n_3-n_2}, b^{m_3 - m_2}}
		\end{equation*}
		and define the linear form
		\begin{equation*}
			\Lambda_{23} := n_2 \log \alpha - m_2 \log b + \log a = \log \frac{a \alpha^{n_2}}{b^{m_2}}.
		\end{equation*}
		Once again we distinguish between two cases and assume first that $ \abs{\Lambda_{23}} > 1 $.
		Then by Lemma \ref{lemma:linbig} we have
		\begin{equation*}
			\frac{3}{5} \leq \abs{e^{\Lambda_{23}} - 1} \leq C_{10} \max \setb{\alpha^{n_3-n_2}, \left( \frac{\alpha}{\abs{\alpha_2}} \right)^{n_3-n_2}, b^{m_3 - m_2}}.
		\end{equation*}
		If the maximum is either $ \alpha^{n_3-n_2} $ or $ \left( \frac{\alpha}{\abs{\alpha_2}} \right)^{n_3-n_2} $, this implies an upper bound $ n_2-n_3 \leq C_{45} $.
		If the maximum is $ b^{m_3 - m_2} $, then the inequality implies
		\begin{equation*}
			\log b \leq (m_3-m_2) \log b \leq C_{46}.
		\end{equation*}
		In both situations we are done.
		Therefore we will now assume that $ \abs{\Lambda_{23}} \leq 1 $.
		
		As we have $ \abs{\Lambda_{12}} \leq 1 $ as well as $ \abs{\Lambda_{23}} \leq 1 $, we can apply Lemma \ref{lemma:linsmall} to both linear forms, which yields
		\begin{align*}
			\abs{\Lambda_{12}} &\leq 4 \abs{e^{\Lambda_{12}} - 1} \leq C_{47} \max \setb{\alpha^{n_2-n_1}, \left( \frac{\alpha}{\abs{\alpha_2}} \right)^{n_2-n_1}, b^{m_2 - m_1}}, \\
			\abs{\Lambda_{23}} &\leq 4 \abs{e^{\Lambda_{23}} - 1} \leq C_{47} \max \setb{\alpha^{n_3-n_2}, \left( \frac{\alpha}{\abs{\alpha_2}} \right)^{n_3-n_2}, b^{m_3 - m_2}}.
		\end{align*}
		As the next step we define a further linear form by
		\begin{equation*}
			\Lambda := m_2 \Lambda_{12} - m_1 \Lambda_{23} = (m_2 n_1 - m_1 n_2) \log \alpha + (m_2 - m_1) \log a.
		\end{equation*}
		Using the upper bounds for $ \abs{\Lambda_{12}} $ and $ \abs{\Lambda_{23}} $ from above we get
		\begin{align}
			\abs{\Lambda} &\leq m_2 \abs{\Lambda_{12}} + m_1 \abs{\Lambda_{23}} \nonumber \\
			\label{eq:s5bound}
			&\leq C_{48} n_1 \max \left( \alpha^{n_2-n_1}, \left( \frac{\alpha}{\abs{\alpha_2}} \right)^{n_2-n_1}, b^{m_2 - m_1}, \right. \\
			&\hspace{4cm} \left. \alpha^{n_3-n_2}, \left( \frac{\alpha}{\abs{\alpha_2}} \right)^{n_3-n_2}, b^{m_3 - m_2} \right). \nonumber
		\end{align}
		
		We aim for applying Proposition \ref{prop:Matveev} to get a lower bound.
		Therefore let us set $ t=2 $, $ K = \QQ(a,\alpha) $, $ D = [K:\QQ] $ as well as 
\begin{align*}
\eta_1 &= \alpha, \quad &
\eta_2 &= a,\\
b_1 &= m_2 n_1 - m_1 n_2, \quad &
b_2 &= m_2 - m_1.
\end{align*}		
		Further we can put
		\begin{align*}
			A_1 &= \max \setb{Dh(\alpha), \log \alpha, 0.16}, \\
			A_2 &= \max \setb{Dh(a), \abs{\log a}, 0.16}, \\
			B &= n_1^2.
		\end{align*}
		Finally, we have to show that $ \Lambda $ is nonzero.
		But this follows immediately from $ m_1 \neq m_2 $ and the assumption that $ \alpha $ and $ a $ are multiplicatively independent.
		Now Proposition \ref{prop:Matveev} states that
		\begin{equation*}
			\log \abs{\Lambda} \geq -C_{49} (1 + \log (n_1^2)) \geq -C_{50} \log (n_1^2) \geq -C_{51} \log n_1.
		\end{equation*}
		
		Comparing this lower bound with the upper bound \eqref{eq:s5bound} gives us
		\begin{align*}
			-C_{51} \log n_1 &\leq \log C_{48} + \log n_1 + \log \max \left( \alpha^{n_2-n_1}, \left( \frac{\alpha}{\abs{\alpha_2}} \right)^{n_2-n_1}, b^{m_2 - m_1}, \right. \\
			&\hspace{5cm} \left. \alpha^{n_3-n_2}, \left( \frac{\alpha}{\abs{\alpha_2}} \right)^{n_3-n_2}, b^{m_3 - m_2} \right)
		\end{align*}
		which immediately implies
		\begin{equation}
			\label{eq:s5minbound}
			\min \setb{n_1-n_2, n_2-n_3, (m_1-m_2) \log b, (m_2-m_3) \log b} \leq C_{52} \log n_1.
		\end{equation}
		We have now to handle three cases.
		If the minimum in Inequality \eqref{eq:s5minbound} is either $ (m_1-m_2) \log b $ or $ (m_2-m_3) \log b $, then we get
		\begin{equation*}
			\log b \leq (m_i-m_{i+1}) \log b \leq C_{52} \log n_1
		\end{equation*}
		for an $ i \in \set{1,2} $ and are done.
		If the minimum in Inequality \eqref{eq:s5minbound} is $ n_2-n_3 $, we have case (ii) of the lemma.
		So it remains to consider the case when the minimum in Inequality \eqref{eq:s5minbound} is $ n_1-n_2 $.
		This can be handled together with the still open case $ n_1-n_2 \leq C_{45} $ from above.
		Hence let us assume that $ n_1-n_2 \leq C_{53} \log n_1 $.
		By Lemma \ref{lemma:step4} we get
		\begin{align*}
			C_{53} \log n_1 &\geq n_1-n_2 \\
			&\geq \frac{1}{\log \alpha} (m_1 \log b - C_{42}) - \frac{1}{\log \alpha} (m_2 \log b + C_{41}) \\
			&= \frac{1}{\log \alpha} \left( (m_1-m_2) \log b - C_{41} - C_{42} \right)
		\end{align*}
		and thus
		\begin{equation*}
			\log b \leq (m_1-m_2) \log b \leq C_{54} \log n_1
		\end{equation*}
		which concludes the proof of the lemma.
	\end{proof}
	
	\begin{mylemma}
		\label{lemma:step6}
		Assume that Equation \eqref{eq:centraleq} has at least three distinct solutions $ (n_1,m_1) $, $ (n_2,m_2) $ and $ (n_3,m_3) $ with $ n_1 > n_2 > n_3 $ as considered in Theorem \ref{thm:mainthm}.
		Then we have
		\begin{equation*}
			\log b \leq C_{40} (\log n_1)^2.
		\end{equation*}
	\end{mylemma}
	
	\begin{proof}
		From Lemma \ref{lemma:step5} we get that either (i) or (ii) holds.
		Since in the case (i) there is nothing to do, we may assume that we are in the case (ii) and therefore have the bound
		\begin{equation}
			\label{eq:s6bound}
			n_2-n_3 \leq C_{44} \log n_1.
		\end{equation}
		Recall the linear form
		\begin{equation*}
			\Lambda_{12} = n_1 \log \alpha - m_1 \log b + \log a
		\end{equation*}
		from the proof of Lemma \ref{lemma:step5}. Note that it is enough to consider the situation $ \abs{\Lambda_{12}} \leq 1 $, which by Lemma \ref{lemma:linsmall} implied
		\begin{equation*}
			\abs{\Lambda_{12}} \leq 4 \abs{e^{\Lambda_{12}} - 1} \leq C_{47} \max \setb{\alpha^{n_2-n_1}, \left( \frac{\alpha}{\abs{\alpha_2}} \right)^{n_2-n_1}, b^{m_2 - m_1}},
		\end{equation*}
		since the situation $ \abs{\Lambda_{12}} > 1 $ led to case (i).
		
		Taking a look at the proof of Lemma \ref{lemma:step2} and noting that we are working with three solutions, we recall from \eqref{eq:s2c1upper} the upper bound 
		\begin{equation*}
			\abs{\frac{a \alpha^{n_3} (\alpha^{n_2-n_3} - 1)}{b^{m_2}} - 1} \leq C_{18} \max \setb{\left( \frac{\abs{\alpha_2}}{\alpha} \right)^{n_2}, b^{m_3-m_2}}
		\end{equation*}
		and define the linear form
		\begin{equation*}
			\widetilde{\Lambda_{23}} := n_3 \log \alpha - m_2 \log b + \log (a(\alpha^{n_2-n_3} - 1)).
		\end{equation*}
		Again we distinguish between two cases and assume first that $ \abs{\widetilde{\Lambda_{23}}} > 1 $.
		Then by Lemma \ref{lemma:linbig} we have
		\begin{equation*}
			\frac{3}{5} \leq \abs{e^{\widetilde{\Lambda_{23}}} - 1} \leq C_{18} \max \setb{\left( \frac{\abs{\alpha_2}}{\alpha} \right)^{n_2}, b^{m_3-m_2}}.
		\end{equation*}
		If the maximum is $ \left( \frac{\abs{\alpha_2}}{\alpha} \right)^{n_2} $, this implies an upper bound $ n_2 \leq C_{55} $.
		We will come back to this later.
		If the maximum is $ b^{m_3 - m_2} $, then the inequality implies
		\begin{equation*}
			\log b \leq (m_2-m_3) \log b \leq C_{56}
		\end{equation*}
		and we are done.
		Therefore we will now assume that $ \abs{\widetilde{\Lambda_{23}}} \leq 1 $.
		In this situation we can apply Lemma \ref{lemma:linsmall} which yields
		\begin{equation*}
			\abs{\widetilde{\Lambda_{23}}} \leq 4 \abs{e^{\widetilde{\Lambda_{23}}} - 1} \leq C_{57} \max \setb{\left( \frac{\abs{\alpha_2}}{\alpha} \right)^{n_2}, b^{m_3-m_2}}.
		\end{equation*}
		
		In the next step we define a further linear form by
		\begin{align*}
			\Lambda :=\ &m_2 \Lambda_{12} - m_1 \widetilde{\Lambda_{23}} \\
			=\ &(m_2 n_1 - m_1 n_3) \log \alpha + (m_2 - m_1) \log a - m_1 \log (\alpha^{n_2-n_3}-1).
		\end{align*}
		Using the upper bounds for $ \abs{\Lambda_{12}} $ and $ \abs{\widetilde{\Lambda_{23}}} $ from above we get
		\begin{align}
			\abs{\Lambda} &\leq m_2 \abs{\Lambda_{12}} + m_1 \abs{\widetilde{\Lambda_{23}}} \nonumber \\
			\label{eq:s6linbound}
			&\leq C_{58} n_1 \max \setb{\alpha^{n_2-n_1}, \left( \frac{\alpha}{\abs{\alpha_2}} \right)^{n_2-n_1}, b^{m_2 - m_1}, \left( \frac{\abs{\alpha_2}}{\alpha} \right)^{n_2}, b^{m_3-m_2}}.
		\end{align}
		
		We aim for applying Proposition \ref{prop:Matveev} to get a lower bound.
		Therefore let us set $ t=3 $, $ K = \QQ(a,\alpha) $, $ D = [K:\QQ] $ as well as
\begin{align*}
\eta_1 &= \alpha, \quad &
\eta_2 &= a, \quad &
\eta_3 &= \alpha^{n_2-n_3}-1,\\
b_1 &= m_2 n_1 - m_1 n_3, \quad &
b_2 &= m_2 - m_1, \quad &
b_3 &= -m_1.
\end{align*}		
		Using the bound \eqref{eq:s6bound} gives us
		\begin{align*}
			h(\eta_3) &\leq h(\alpha^{n_2-n_3}) + h(1) + \log 2 = (n_2-n_3) h(\alpha) + h(1) + \log 2 \\
			&\leq C_{59} (n_2-n_3) \leq C_{60} \log n_1
		\end{align*}
		and an analogous bound for $ \abs{\log \eta_3} $.
		Thus we can put
		\begin{align*}
			A_1 &= \max \setb{Dh(\alpha), \log \alpha, 0.16}, \\
			A_2 &= \max \setb{Dh(a), \abs{\log a}, 0.16}, \\
			A_3 &= C_{61} \log n_1, \\
			B &= n_1^2.
		\end{align*}
		Finally, we have to show that $ \Lambda $ is nonzero.
		Assume the contrary.
		Since $ m_1 \neq m_2 $, this means that $ \alpha $, $ a $ and $ \alpha^{n_2-n_3}-1 $ are multiplicatively dependent.
		Thus there exist rational numbers $ x,y \in \QQ $ such that
		\begin{equation*}
			\alpha^{n_2-n_3} - 1 = a^x \alpha^y
		\end{equation*}
		because $ \alpha $ and $ a $ are multiplicatively independent by assumption.
		Then the linear form $ \Lambda $ becomes
		\begin{equation*}
			\Lambda = (m_2 n_1 - m_1 n_3 - m_1 y) \log \alpha + (m_2 - m_1 - m_1 x) \log a.
		\end{equation*}
		Again using that $ \alpha $ and $ a $ are multiplicatively independent by assumption, $ \Lambda = 0 $ in particular implies
		\begin{equation*}
			m_2 - m_1 - m_1 x = 0
		\end{equation*}
		which yields
		\begin{equation*}
			m_2 = m_1 (1+x)
		\end{equation*}
		and hence $ -1 < x < 0 $ since $ m_1 > m_2 \geq 1 $.
		But this was excluded in the theorem.
		Therefore we have $ \Lambda \neq 0 $.
		Now Proposition \ref{prop:Matveev} states that
		\begin{equation*}
			\log \abs{\Lambda} \geq -C_{62} (1 + \log (n_1^2)) \log n_1 \geq -C_{63} \log (n_1^2) \log n_1 \geq -C_{64} (\log n_1)^2.
		\end{equation*}
		
		Comparing this lower bound with the upper bound \eqref{eq:s6linbound} gives us
		\begin{align*}
			-C_{64} (\log n_1)^2 &\leq \log C_{58} + \log n_1 \\
			&\hspace{0.5cm}+ \log \max \setb{\alpha^{n_2-n_1}, \left( \frac{\alpha}{\abs{\alpha_2}} \right)^{n_2-n_1}, b^{m_2 - m_1}, \left( \frac{\abs{\alpha_2}}{\alpha} \right)^{n_2}, b^{m_3-m_2}}
		\end{align*}
		which immediately implies
		\begin{equation}
			\label{eq:s6minbound}
			\min \setb{n_1-n_2, n_2, (m_1-m_2) \log b, (m_2-m_3) \log b} \leq C_{65} (\log n_1)^2.
		\end{equation}
		We have now to handle three cases.
		If the minimum in Inequality \eqref{eq:s6minbound} is either $ (m_1-m_2) \log b $ or $ (m_2-m_3) \log b $, then we get
		\begin{equation*}
			\log b \leq (m_i-m_{i+1}) \log b \leq C_{65} (\log n_1)^2
		\end{equation*}
		for an $ i \in \set{1,2} $ and are done.
		If the minimum in Inequality \eqref{eq:s6minbound} is $ n_1-n_2 $, then we have
		\begin{equation*}
			n_1-n_2 \leq C_{65} (\log n_1)^2
		\end{equation*}
		which yields, by using Lemma \ref{lemma:step4}, analogously to the end of the proof of Lemma \ref{lemma:step5} the bound
		\begin{equation*}
			\log b \leq (m_1-m_2) \log b \leq C_{66} (\log n_1)^2
		\end{equation*}
		and we are done as well.
		So it remains to consider the case when the minimum in Inequality \eqref{eq:s6minbound} is $ n_2 $.
		This can be handled together with the still open case $ n_2 \leq C_{55} $ from above.
		Hence let us assume that $ n_2 \leq C_{67} (\log n_1)^2 $.
		By Lemma \ref{lemma:step4} we get
		\begin{equation*}
			C_{67} (\log n_1)^2 \geq n_2 \geq \frac{1}{\log \alpha} (m_2 \log b - C_{42}) \geq \frac{1}{\log \alpha} (\log b - C_{42})
		\end{equation*}
		and thus
		\begin{equation*}
			\log b \leq C_{68} (\log n_1)^2
		\end{equation*}
		which concludes the proof of the lemma.
	\end{proof}
	
	We have now reached our second milestone, the bound \eqref{eq:bbound} is proven.
	So if there are at least three distinct solutions $ (n_1,m_1) $, $ (n_2,m_2) $ and $ (n_3,m_3) $ with $ n_1 > n_2 > n_3 $ as considered in Theorem \ref{thm:mainthm}, then we have the two bounds \eqref{eq:logbound} and \eqref{eq:bbound}.
	Inserting \eqref{eq:bbound} into \eqref{eq:logbound} gives
	\begin{equation*}
		n_1 \leq C_{69} (\log n_1)^9
	\end{equation*}
	and therefore the absolute bound
	\begin{equation*}
		n_1 \leq C_{70}.
	\end{equation*}
	Now we insert this into Inequality \eqref{eq:bbound} and get
	\begin{equation*}
		b \leq C_{71}.
	\end{equation*}
	This proves Theorem \ref{thm:mainthm}.
	
\section{Proof of Theorem \ref{thm:Tribos}}\label{sec:Tribos}\label{sec:proofTribos}

Recall that the Tribonacci sequence is given by  $T_1=1, T_2=1, T_3=2$ and $T_{n}=T_{n-1}+T_{n-2}+T_{n-3}$ for $n\geq 4$.
Then one can compute the roots $\alpha,\beta,\gamma$ of $f(X)=X^3-X^2-X-1$ and the coefficients $a,b,c$ such that
\[
	T_n= a\alpha^n + b\beta^n + c\gamma^n.
\]
Let $\sigma$ be the Galois automorphism on the splitting field $K$ of $f$ that maps $\alpha \mapsto \beta$. It turns out that
\[
	a=\frac{1}{-\alpha^2 + 4\alpha -1}, \quad
	b=\sigma(a),\quad
	c=\sigma^2(a).
\]
We fix $\alpha,\beta, \gamma$ such that
\begin{align*}
	\alpha &\approx 1.839, \quad &
	\beta &\approx -0.42 + 0.61 i, \quad &
	\gamma &\approx -0.42 - 0.61 i,\\
	a &\approx 0.336, \quad &
	b &\approx -0.17 - 0.20 i, \quad &
	c &\approx -0.17 + 0.20i.
\end{align*}
Note that 
\[	
	|\beta|=|\gamma|\leq 0.74 
	\quad \text{and} \quad 
	|b|=|c|\leq 0.26, 
\]	
so we can estimate 
\begin{equation}\label{eq:trib:Tn}	
	T_n= a \alpha^n + L(0.52 \cdot 0.74^n).
\end{equation}
Here the $L$-notation means the following: For functions $f(n)$, $g(n)$ with $g(n)>0 $ for $n\geq 1$ we write
\[
	f(n) = L(g(n))
	\quad \text{if} \quad
	|f(n)|\leq g(n).
\]

We check that all assumptions in Theorem \ref{thm:mainthm} are fulfilled: First, $\alpha$ is indeed an irrational dominant root larger than 1 and $a>0$. 
Second, we check hat $a$ and $\alpha$ are multiplicatively independent.
Assume that they are not, then there exist nonzero integers $x,y$ such that $a^x \alpha^y=1$. Since $|a|<1$ and $|\alpha|>1$, the integers $x$ and $y$ must have the same sign. However, if we apply $\sigma$, we get $b^x\beta^y=1$, but since $|b|<1$ and $|\beta|<1$ this is impossible if $x$ and $y$ have the same sign. Therefore, $a$ and $\alpha$ are multiplicatively independent. Finally, we check that there are no unwanted solutions to \eqref{eq:techcond}. Assume that 
\[
	\alpha^z - 1 = a^x \alpha^y
\]	
with $ z \in \NN $, $ x,y \in \QQ $ and $ -1 < x < 0 $.
Then we have for the norms
\[
	N_{K/\QQ}(\alpha^z-1) 
	= N_{K/\QQ}(a^x \alpha^y)
	= N_{K/\QQ}(a)^x N_{K/\QQ}(\alpha)^y
	= (1/44)^x1^y
	=44^{-x}.
\]
Now since $\alpha$ and therefore $\alpha^z -1$ are algebraic integers, $N_{K/\QQ}(\alpha^z-1)=44^{-x}$ has to be an integer. But this is impossible for a rational $x$ between $-1$ and $0$.

Assume that we have three solutions $T_{n_1}-b^{m_1}=T_{n_2}-b^{m_2}=T_{n_3}-b^{m_3}=c$ with $n_1>n_2>n_3\geq 2$.
We do some preliminary estimations.

From \eqref{eq:trib:Tn} we have
\begin{align}\label{eq:trib:eqationL}
	a \alpha^{n_1} - b^{m_1} - (a \alpha^{n_2} - b^{m_2})\nonumber
	&= L(0.52 \cdot 0.74^{n_1}) + L(0.52 \cdot 0.74^{n_2})\\
	&= L(0.91 \cdot 0.74^{n_2})
\end{align}
and analogously
\begin{equation}\label{eq:trib:equationL2}
	a \alpha^{n_2} - b^{m_2} - (a \alpha^{n_3} - b^{m_3})
	= L(0.91 \cdot 0.74^{n_3}).
\end{equation}
Moreover, we have
\begin{multline*}
	b^{m_1} \nonumber
	\geq b^{m_1} - b^{m_2}
	= T_{n_1} - T_{n_2}
	=  a \alpha^{n_1} - a \alpha^{n_2} + L(0.91 \cdot 0.74^{n_2}) \\
	\geq a(1-\alpha^{-1})\alpha^{n_1} - 0.91 \cdot 0.74^{n_2}
	\geq 0.15 \alpha^{n_1} - 0.68.
\end{multline*}
Now note that on the one hand for $n_1 \geq 6$ we have $0.68/\alpha^{n_1} \leq 0.02$ and on the other hand $b^{m_1}\geq 2^2=4\geq 0.13 \alpha^{n_1}$ is trivially fulfilled for $n_1\leq 5$. Thus we have in any case
\begin{equation}\label{eq:trib:balpha}
	b^{m_1}\geq 0.13 \alpha^{n_1}
	\quad
	\text{and analogously}
	\quad
	b^{m_2}\geq 0.13 \alpha^{n_2}.
\end{equation}
Estimating from the other side, we have
\begin{align*}
	0.5 \alpha^{n_1}
	&\geq a \alpha^{n_1} + L(0.52\cdot 0.74^{n_1})
	= T_{n_1}
	\geq T_{n_1}-T_{n_2}\\
	&= b^{m_1}-b^{m_2}
	\geq 0.5 b^{m_1}
\end{align*}
and the same argument works for the second and third solution. So we have
\begin{equation}\label{eq:trib:alphab}
	\alpha^{n_1} \geq b^{m_1}
	\quad \text{and} \quad
	\alpha^{n_2} \geq b^{m_2}.
\end{equation}
In particular, \eqref{eq:trib:balpha} and \eqref{eq:trib:alphab} imply the bounds
\begin{equation}\label{eq:trib:nlogalphamlogb}
	\begin{split}
	&m_1 \log b
	\leq n_1 \log \alpha
	\leq m_1 \log b + 2.1,\\
	&m_2 \log b
	\leq n_2 \log \alpha
	\leq m_2 \log b + 2.1
	\end{split}
\end{equation}
and in particular
\[
	n_1 \geq m_1.
\]

\step{step:smallSols}{Small solutions:} First, we check that there are no solutions with $n_1\leq 150$. 
To that end we simply search for all differences of two Tribonacci numbers that can be written in the form $T_{n_1}-T_{n_2} = b^{m_1}-b^{m_2}$ with $b\geq 2$ and $m_1>m_2\geq 1$. 
With the help of Sage \cite{sagemath} this is not difficult (see \nameref{sec:appendix} for the code):
For each pair $2\leq n_2 < n_1 \leq 150$, we compute the prime factorisation of $T_{n_1}-T_{n_2}=p_1^{k_1}\cdots p_l^{k_l}$. 
Now we need to check if there exist $b\geq 2$ and $x,y\geq 1$ 
such that $p_1^{k_1}\cdots p_l^{k_l}=b^x(b^y-1)$. 
Since $\gcd(b^x,b^y-1)=1$, we can simply try 
$b=p_{i_1}^{k_{i_1}/d}\cdots p_{i_t}^{k_{i_t}/d}$ for each subset $\{p_{i_1},\ldots p_{i_t}\}\subset\{p_1,\ldots,p_l\}$ and $d=\gcd(k_{i_1},\ldots,k_{i_t})$, and check if $(T_{n_1}-T_{n_2})/b^d$ can be written in the form $(b^y -1)$.
It turns out that this only happens on 14 occasions:
\begin{align*}
T_{ 4 } - T_{ 3 } = 2 ^{ 1 } ( 2 ^{ 1 }-1), \quad c&= 0 = T_{ 4 } - 2 ^{ 2 } = T_{ 3 } - 2 ^{ 1 }; \\
T_{ 5 } - T_{ 2 } = 2 ^{ 1 } ( 2 ^{ 2 }-1), \quad c&= -1 = T_{ 5 } - 2 ^{ 3 } = T_{ 2 } - 2 ^{ 1 }; \\
T_{ 5 } - T_{ 2 } = 3 ^{ 1 } ( 3 ^{ 1 }-1), \quad c&= -2 = T_{ 5 } - 3 ^{ 2 } = T_{ 2 } - 3 ^{ 1 }; \\
T_{ 6 } - T_{ 2 } = 2 ^{ 2 } ( 2 ^{ 2 }-1), \quad c&= -3 = T_{ 6 } - 2 ^{ 4 } = T_{ 2 } - 2 ^{ 2 }; \\
T_{ 6 } - T_{ 5 } = 2 ^{ 1 } ( 2 ^{ 2 }-1), \quad c&= 5 = T_{ 6 } - 2 ^{ 3 } = T_{ 5 } - 2 ^{ 1 }; \\
T_{ 6 } - T_{ 5 } = 3 ^{ 1 } ( 3 ^{ 1 }-1), \quad c&= 4 = T_{ 6 } - 3 ^{ 2 } = T_{ 5 } - 3 ^{ 1 }; \\
T_{ 7 } - T_{ 4 } = 5 ^{ 1 } ( 5 ^{ 1 }-1), \quad c&= -1 = T_{ 7 } - 5 ^{ 2 } = T_{ 4 } - 5 ^{ 1 }; \\
T_{ 8 } - T_{ 3 } = 7 ^{ 1 } ( 7 ^{ 1 }-1), \quad c&= -5 = T_{ 8 } - 7 ^{ 2 } = T_{ 3 } - 7 ^{ 1 }; \\
T_{ 8 } - T_{ 7 } = 5 ^{ 1 } ( 5 ^{ 1 }-1), \quad c&= 19 = T_{ 8 } - 5 ^{ 2 } = T_{ 7 } - 5 ^{ 1 }; \\
T_{ 11 } - T_{ 3 } = 17 ^{ 1 } ( 17 ^{ 1 }-1), \quad c&= -15 = T_{ 11 } - 17 ^{ 2 } = T_{ 3 } - 17 ^{ 1 }; \\
T_{ 12 } - T_{ 4 } = 5 ^{ 3 } ( 5 ^{ 1 }-1), \quad c&= -121 = T_{ 12 } - 5 ^{ 4 } = T_{ 4 } - 5 ^{ 3 }; \\
T_{ 12 } - T_{ 7 } = 2 ^{ 5 } ( 2 ^{ 4 }-1), \quad c&= -8 = T_{ 12 } - 2 ^{ 9 } = T_{ 7 } - 2 ^{ 5 }; \\
T_{ 15 } - T_{ 11 } = 54 ^{ 1 } ( 54 ^{ 1 }-1), \quad c&= 220 = T_{ 15 } - 54 ^{ 2 } = T_{ 11 } - 54 ^{ 1 }; \\
T_{ 23 } - T_{ 12 } = 641 ^{ 1 } ( 641 ^{ 1 }-1), \quad c&= -137 = T_{ 23 } - 641 ^{ 2 } = T_{ 12 } - 641 ^{ 1 }.
\end{align*}
The only $c$ that appears twice is $c=-1$, but on the two occasions the $b$'s are distinct ($2$ and $5$). So there is no $c$ with three solutions with $n_1\leq 150$.
The computations only took a couple of minutes on a usual laptop. 
Let us from now on assume that
\[
	n_1>150.
\]

Finally, a little remark on notation: The constants in this section will be numbered starting with $C_{100}$ to set them apart from the previous constants.

\begin{mystep}\label{step:Step1}
From \eqref{eq:trib:eqationL} we obtain
\[
	|a \alpha^{n_1} - b^{m_1}|
	= |a \alpha^{n_2} - b^{m_2} + L(0.91 \cdot 0.74^{n_2})|
	\leq \max( a \alpha^{n_2}, b^{m_2} )
	\leq \alpha^{n_2},
\]
where for the first inequality we used the fact that $a\alpha^{n_2}$ and $b^{m_2}$ are both positive and larger than the $L$-terms and for the second estimation we used $a<1$ and \eqref{eq:trib:alphab}.
Dividing by $b^{m_1} \geq 0.13 \alpha^{n_1}$ (see \eqref{eq:trib:balpha}) we obtain
\begin{equation}\label{eq:trib:Lambda1}
	|\Lambda_1| 
	:= \left| \frac{a\alpha^{n_1}}{b^{m_1}}-1 \right|
	\leq \frac{\alpha^{n_2}}{0.13 \alpha^{n_1}}
	\leq 7.7 \alpha^{-(n_1-n_2)}.
\end{equation}
We check that $\Lambda_1 \neq 0$: If $\Lambda_1=0$, then $a \alpha^{n_1} =b^{m_1} \in \ZZ$, so an application of the Galois automorphism $\sigma$ does not change $a \alpha^{n_1}$, i.e.\ we get $a \alpha^{n_1} = \sigma (a \alpha^{n_1}) =b \beta^{n_1}$.
Taking absolute values and estimating we get $0.33 \cdot 1.83^{n_1} \leq |a \alpha^{n_1}| = |b \beta^{n_1}|\leq 0.26 \cdot 0.74^{n_1}$,
which is impossible.

Now we apply Proposition~\ref{prop:Matveev} with
$t=3$, $K=\QQ(a,\alpha)=\QQ(\alpha)$, $D=3=[K:\QQ]$ and
\begin{align*}
\eta_1 &= \alpha, \quad &
\eta_2 &= b, \quad &
\eta_3 &=a, \\
b_1 &= n_1, \quad &
b_2 &= -m_1, \quad &
b_3 &= 1.
\end{align*}
We put
\begin{align*}
	A_1&=0.7\geq \log \alpha = \max(D h(\alpha),\log \alpha,0.16),\\
	A_2&=3 \log b= \max(D h(b),\abs{\log b},0.16),\\
	A_3&= 3.8 \geq \log 44 =3 h(a)= \max(D h(a),\abs{\log a},0.16),\\
	B&=n_1.
\end{align*}
Thus we get that
\[
	\log|\Lambda_1|
	\geq - C_{100} \cdot \log n_1 \cdot \log b,
\]
where
\[
	C_{100} = 2.6 \cdot 10^{13}
	> 1.4 \cdot 30^6 \cdot 3^{4.5} \cdot 3^2 (1+ \log 3) \cdot 1.2 \cdot 0.7 \cdot 3 \cdot 3.8.
\]
Note that the factor $1.2$ comes from the estimate $1+\log B = 1+\log n_1 \leq 1.2 \log n_1$, since $n_1>150$.
Together with \eqref{eq:trib:Lambda1} this implies
\[
	-C_{100}\cdot \log n_1 \cdot \log b
	\leq \log 7.7 - (n_1-n_2)\log \alpha,
\]
which yields
\begin{equation}\label{eq:trib:boundn1-n2}
	(n_1-n_2)\log \alpha
	\leq C_{100}\cdot \log n_1 \cdot \log b.
\end{equation}
Here we omitted the term $\log 7.7$. We can do this because we estimated quite generously when computing the constant $C_{100}=2.6\cdot 10^{13}$.
\end{mystep}

\begin{mystep}{}\label{step:Step2}
From \eqref{eq:trib:eqationL} we obtain
\[
	\left| (a \alpha^{n_1} - a \alpha^{n_2}) - (b^{m_1} - b^{m_2}) \right|
	= L(0.91 \cdot 0.74^{n_2})
	\leq 0.5.
\]
Dividing by $b^{m_1} - b^{m_2} \geq 0.5 b^{m_1}$ we obtain
\begin{equation}\label{eq:trib:Lambda2}
	|\Lambda_2|
	:= \left| \frac{a \alpha^{n_2}(\alpha^{n_1-n_2}-1)}{b^{m_2} (b^{m_1-m_2}-1)}-1\right|
	\leq b^{-m_1}.
\end{equation}
We check that $\Lambda_2 \neq 0$: If $\Lambda_{2}=0$, then $a (\alpha^{n_1}-\alpha^{n_2}) \in \ZZ$, so we must have $a (\alpha^{n_1}-\alpha^{n_2})= \sigma (a (\alpha^{n_1}-\alpha^{n_2})) = b(\beta^{n_1}-\beta^{n_2})$. Taking absolute values and estimating we get
\begin{align*}
	0.15 \cdot 1.83^{n_1}	
	&\leq 0.33 \alpha^{n_1} (1-\alpha^{-1})
	\leq \left| a (\alpha^{n_1}-\alpha^{n_2}) \right|
	= \left| b (\beta^{n_1}-\beta^{n_2}) \right|\\
	&\leq 0.26 (|\beta|^{n_1} + |\beta|^{n_2})
	\leq 0.26 (0.74^{n_1} + 0.74^{n_2}),
\end{align*}
which is impossible for $n_1>150$.

Now we apply Proposition~\ref{prop:Matveev} just like in Step \ref{step:Step1}, except that now $b_1=n_2$, $b_2=-m_2$ and $\eta_3=\frac{a(\alpha^{n_1-n_2}-1)}{b^{m_1-m_2}-1}$. In order to find an $A_3$ we estimate the height of $\eta_3$:
\begin{align*}
h\left( \frac{a(\alpha^{n_1-n_2}-1)}{b^{m_1-m_2}-1} \right) 
	&\leq h(a) + h(\alpha^{n_1-n_2}-1) + h(b^{m_1-m_2}-1)\\
	&\leq \frac{1}{3} \log 44 + (n_1-n_2)h(\alpha) + \log 2 + \log (b^{m_1-m_2}-1)\\
	&\leq 1.96 + (n_1-n_2) \frac{\log \alpha}{3} + (m_1-m_2)\log b\\
	&\leq 1.96 + (n_1-n_2) \frac{\log \alpha}{3} +  (n_1-n_2)\log \alpha + 2.1\\
	&\leq 4.9 (n_1-n_2),
\end{align*}
where we used \eqref{eq:trib:nlogalphamlogb} to estimate $(m_1-m_2)\log b \leq (n_1-n_2)\log \alpha + 2.1$.

Thus we can set $A_3= 14.7 (n_1-n_2) \geq \max(D h(\eta_3), \abs{\log \eta_3}, 0.16)$ and we obtain analogously to the application of Proposition~\ref{prop:Matveev} in Step~\ref{step:Step1}
\[
	\log|\Lambda_2|
	\geq - C_{101} \cdot \log n_1 \cdot \log b \cdot (n_1-n_2),
\]
where
\[
	C_{101} = 1.1 \cdot 10^{14}
	> 1.4 \cdot 30^6 \cdot 3^{4.5} \cdot 3^2 (1+ \log 3) \cdot 1.2 \cdot 0.7 \cdot 3 \cdot 14.7.
\]
Together with \eqref{eq:trib:Lambda2} this implies
\[
	m_1 \log b 
	\leq C_{101} \cdot \log n_1 \cdot \log b \cdot (n_1-n_2).
\]
Using \eqref{eq:trib:nlogalphamlogb} and \eqref{eq:trib:boundn1-n2} from Step~\ref{step:Step1}, we obtain
\begin{multline*}
	n_1 \log \alpha
	\leq m_1 \log b + 2.1
	\leq C_{101} \cdot \log n_1 \cdot \log b \cdot (n_1-n_2)\\
	\leq C_{101} \cdot \log n_1 \cdot \log b \cdot (\log \alpha)^{-1} \cdot C_{100}\cdot \log n_1 \cdot \log b,
\end{multline*}
where we omitted the constant $ 2.1 $ because $C_{101}$ was estimated roughly.
Thus, we end up with
\begin{equation}\label{eq:trib:boundStep2}
	n_1 \leq C_{102} \cdot (\log n_1 \cdot \log b)^2,
\end{equation}
where
\[
	C_{102} = 7.8 \cdot 10^{27} > C_{100}\cdot C_{101} \cdot (\log \alpha)^{-2
}.
\]
\end{mystep}

\begin{mystep}\label{step:Step3}
Recall the bound \eqref{eq:trib:Lambda1}:
\[
	|\Lambda_1|
	=\left| \frac{a\alpha^{n_1}}{b^{m_1}}-1 \right|
	\leq 7.7 \alpha^{-(n_1-n_2)}.
\]
Assume for a moment that $|\Lambda_1|\geq 0.5$. Then $n_1-n_2 \leq 4$, we get that $\log b \leq m_1 \log b - m_2 \log b \leq (n_1-n_2)\log \alpha + 2.1 \leq 4.6$ and we can immediately skip to Step~\ref{step:Step6}.

Therefore, we may assume that $|\Lambda_1|< 0.5$. Since $\abs{\log x} \leq 2 |x-1|$ for $|x-1|<0.5$, we obtain
\[
	|\Lambda_1'|
	:= \abs{\log a + n_1 \log \alpha - m_1 \log b}
	\leq 15.4 \alpha^{-(n_1-n_2)}.
\]

Now we consider the second and the third solution and obtain from \eqref{eq:trib:equationL2}
\[
	|a \alpha^{n_2}-b^{m_2}|
	=|a\alpha^{n_3}-b^{m_3}+L(0.91\cdot 0.74^{n_3})|
	\leq \max(a \alpha^{n_3}, b^{m_3}).
\]
Dividing by $b^{m_2} \geq 0.13 \alpha^{n_2}$ (by \eqref{eq:trib:balpha})
we obtain 
\[
	|\Lambda_{12}|
	:=\left| \frac{a\alpha^{n_2}}{b^{m_2}}-1 \right|
	\leq \max\left( \frac{a \alpha^{n_3}}{0.13 \alpha^{n_2}}, \frac{b^{m_3}}{b^{m_2}}\right)
	\leq \max(2.6 \alpha^{-(n_2-n_3)}, b^{-(m_2-m_3)}).
\]
Assume for a moment that $|\Lambda_{12}|\geq 0.5$. Then we either get $n_2-n_3\leq 2$ or $b=2$. In the first case we can immediately skip to the next step. In the second case we can immediately skip to Step~\ref{step:Step6}. Thus we may assume $|\Lambda_{12}|<0.5$ and we get that
\[
	|\Lambda_{12}'|
	:= \abs{\log a + n_2 \log \alpha - m_2 \log b}
	\leq \max(5.2 \alpha^{-(n_2-n_3)}, 2b^{-(m_2-m_3)}).
\]
Thus for the linear form $\Lambda_3':= m_2 \Lambda_1' - m_1 \Lambda_{12}'$ we get the upper bound
\begin{align}\label{eq:trib:Lambda3}
	|\Lambda_{3}'|
	&=|(n_1m_2 - n_2m_1)\log \alpha + (m_2-m_1)\log a| \nonumber\\
	&\leq m_2 \cdot 15.4 \alpha^{-(n_1-n_2)} +
		m_1 \cdot \max(5.2 \alpha^{-(n_2-n_3)}, 2b^{-(m_2-m_3)}) \nonumber\\
	&\leq 20.6 n_1 \max( \alpha^{-(n_1-n_2)}, \alpha^{-(n_2-n_3)},b^{-(m_2-m_3)} ).
\end{align} 

Now we have a linear form in only two logarithms, so we can use Laurent's bound instead of Matveev's.

We set
\begin{align*}
\eta_1 &= \alpha, \quad b_1 = n_1m_2 - n_2 m_1,\\
\eta_2 &= a, \quad b_2 = m_2-m_1,\\
D &= 3, \\
\log A_1 &= 1 = \max(3h(\alpha),\log \alpha,1),\\
\log A_2 &= 3.8 \geq \log 44 = 3h(a) = \max(Dh(a),\abs{\log a}, 1),\\
b' &= \frac{|n_1 m_2 - n_2 m_1|}{3.8} + \frac{|m_2-m_1|}{1}
	\leq {n_1^2}/{3.7}.
\end{align*}
We estimate the factor
\[
	\max( \log b' + 0.38, 18/D, 1)
	\leq \max (2 \log n_1 - \log 3.7 + 0.38, 6, 1)
	\leq 2 \log n_1.
\]
In order to apply Proposition \ref{prop:Laurent}, we need to check if $a$ and $\alpha$ are multiplicatively independent, and if $b_1$ and $b_2$ are nonzero.
We have already checked at the beginning of this section that $a$ and $\alpha$ are multiplicatively independent and we are assuming that $n_1>n_2$, which implies $m_1>m_2$, so $b_2=m_2-m_1$ is nonzero. Assume for a moment that $b_1=0$. Then we have
\begin{align*}
	1
	<\abs{\log a}
	\leq |b_2 \log a|
	=|\Lambda_3'|,	
\end{align*}
so
\[
	\log |\Lambda_3'|
	>0,
\]
which is much better than what we will obtain from the application of Proposition~\ref{prop:Laurent}. So let us now apply Proposition \ref{prop:Laurent}. We obtain
\[
	\log |\Lambda_{3}'|
	\geq -C_{103} (\log n_1)^2, 
\]
where
\[
	C_{103}
	= 2778
	> 20.3 \cdot 3^2 \cdot 2^2 \cdot 1 \cdot 3.8.
\]
Together with \eqref{eq:trib:Lambda3} this implies
\begin{multline*}
	- 2778 \cdot (\log n_1)^2 \\
	\leq \log 20.6 + \log n_1 - \min((n_1-n_2)\log \alpha, (n_2-n_3)\log \alpha, (m_2-m_3)\log b) 
\end{multline*}
and we get
\[
	\min((n_1-n_2)\log \alpha, (n_2-n_3)\log \alpha, (m_2-m_3)\log b)
	\leq 2780 \cdot (\log n_1)^2.
\]
If the minimum is realised by $(n_1-n_2)\log \alpha$, then we can immediately skip to Step~\ref{step:Step5}. If it is realised by $(n_2-n_3)\log \alpha$, then we have
\begin{equation}\label{eq:trib:boundn2-n3}
	n_2-n_3 \leq 4563 \cdot (\log n_1)^2
\end{equation}
and go to the next step. If it is realised by $(m_2-m_3)\log b$, we get that $\log b \leq (m_2-m_3)\log b \leq 2780 \cdot (\log n_1)^2$ and we can skip to Step~\ref{step:Step6}.
\end{mystep}

\begin{mystep}\label{step:Step4}
First, recall from the previous step the bound
\[
	|\Lambda_1'|
	= \abs{\log a + n_1 \log \alpha - m_1 \log b}
	\leq 15.4 \alpha^{-(n_1-n_2)}.
\]
Second, we generate a new linear form in logarithms by starting from \eqref{eq:trib:equationL2}:
\[
	|a \alpha^{n_3} (\alpha^{n_2-n_3} - 1) - b^{m_2}| 
	= |b^{m_3} + L(0.91\cdot 0.74^{n_3})|
	\leq b^{m_3} + 0.7.
\]
Dividing by $b^{m_2}$ we obtain
\[	
	|\Lambda_{41}|
	:=\abs{\frac{a \alpha^{n_3} (\alpha^{n_2-n_3} - 1)}{b^{m_2}} - 1}
	\leq b^{-(m_2-m_3)} + 0.7 b^{-m_2}
	\leq 1.7 b^{-(m_2-m_3)}.
\]
If $|\Lambda_{41}|\geq 0.5$, then we immediately get $b\leq 3$ and we can skip to Step~\ref{step:Step6}. Let us assume $|\Lambda_{41}|< 0.5$. Then we obtain
\[
	|\Lambda_{41}'|
	:= \abs{\log a + n_3 \log \alpha + \log (\alpha^{n_2-n_3}-1) - m_2 \log b}
	\leq 3.4 b^{-(m_2-m_3)}.
\]
Hence we have for the linear form $\Lambda_4':= m_2 \Lambda_1' -m_1 \Lambda_{41}'$ the upper bound
\begin{align}\label{eq:trib:Lambda4}
	|\Lambda_4'|
	&= |(m_2-m_1)\log a + (m_2n_1 - m_1n_3) \log \alpha - m_1 \log(\alpha^{n_2-n_3}-1)| \nonumber\\
	&\leq m_2 \cdot 15.4 \alpha^{-(n_1-n_2)} + m_1 \cdot 3.4 b^{-(m_2-m_3)} \nonumber\\
	&\leq 18.8 \cdot n_1 \max( \alpha^{-(n_1-n_2)}, b^{-(m_2-m_3)}).
\end{align}
Now $\Lambda_4'\neq 0$ because the technical condition involving Equation \eqref{eq:techcond} is fulfilled, and we can apply Proposition \ref{prop:Matveev}. We set $t=3$, $K=\QQ(a,\alpha)=\QQ(\alpha)$, $D=3=[K:\QQ]$, as well as
\begin{align*}
	\eta_1&=a, \quad &
	\eta_2&=\alpha, \quad &
	\eta_3&=\alpha^{n_2-n_3}-1,\\
	b_1 &=m_2-m_1, \quad &
	b_2 &=m_2n_1 - m_1n_3, \quad &
	b_3 &= - m_1.
\end{align*}
Further, we can put
\begin{align*}
	A_1&=3.8 \geq \log 44= \max(Dh(a),\abs{\log a}, 0.16),\\
	A_2&=0.7 \geq \log \alpha =\max(Dh(\alpha),\log \alpha, 0.16),\\
	A_3&= 2.7 (n_2-n_3)
		\geq 3 ( (n_2-n_3)h(\alpha)+ \log 2)
		\geq \max(Dh(\eta_3),\abs{\log\eta_3}, 0.16),\\
	B&= n_1^2.
\end{align*}
Thus we get the lower bound
\[
	\log |\Lambda_4'|
	\geq - C_{104} (n_2-n_3) \log n_1,
\]
where
\[
	C_{104}
	= 4.3 \cdot 10^{13}
	> 1.4 \cdot 30^6 \cdot 3^{4.5} \cdot 3^2 \cdot (1+ \log 3) \cdot 2.2 \cdot 3.8 \cdot 0.7 \cdot 2.7.
\]
Note that we used the estimate $1+ \log B = 1+ \log (n_1^2) \leq 2.2 \log n_1$ for $n_1 > 150$. 

Together with \eqref{eq:trib:Lambda4} this implies 
\[
	- C_{104} (n_2-n_3) \log n_1
	\leq \log 18.8 + \log n_1 - \min ((n_1-n_2) \log \alpha, (m_2-m_3)\log b),
\]
from which we get, omitting small terms because $C_{104}$ was roughly estimated,
\[
	\min ((n_1-n_2) \log \alpha, (m_2-m_3)\log b)
	\leq C_{104} (n_2-n_3) \log n_1.
\]
If the minimum is realised by $(n_1-n_2) \log \alpha$, then using \eqref{eq:trib:boundn2-n3} from the previous step we obtain
\begin{align}\label{eq:trib:boundStep5}
	n_1-n_2
	&\leq (\log \alpha)^{-1} \cdot C_{104} (n_2-n_3) \log n_1 \nonumber\\
	& \leq (\log \alpha)^{-1} \cdot C_{104} \cdot \log n_1 \cdot 4563 \cdot (\log n_1)^2 \nonumber\\
	& \leq C_{105} (\log n_1)^3,
\end{align}
with
\[
	C_{105}
	= 3.3\cdot 10^{17}
	> (\log \alpha)^{-1} \cdot C_{104} \cdot 4563.
\]
If the minimum is realised by $(m_2-m_3)\log b$, then in a similar way we obtain 
\begin{align*}
	\log b
	\leq (m_2-m_3)\log b
	\leq C_{104} \cdot \log n_1 \cdot 4563 \cdot (\log n_1)^2
	\leq 2 \cdot 10^{17} \cdot (\log n_1)^3
\end{align*}
and we can skip to Step \ref{step:Step6}.
\end{mystep}

\begin{mystep}\label{step:Step5}
We now use \eqref{eq:trib:nlogalphamlogb} and \eqref{eq:trib:boundStep5} to obtain a bound for $\log b$:
\begin{multline*}
	\log b 
	\leq (m_1-m_2)\log b
	= m_1 \log b- m_2 \log b\\
	\leq n_1 \log \alpha - n_2 \log \alpha + 2.1
	= (n_1-n_2)\log \alpha + 2.1
	\leq C_{105} (\log n_1)^3 \log \alpha,
\end{multline*}
where we omitted the constant $ 2.1 $ because $C_{105}$ came from a rough estimation.
Thus we have
\begin{equation}\label{eq:trib:boundStep6}
	\log b 
	\leq C_{106} (\log n_1)^3,
\end{equation}
with
\[
	C_{106} 
	= 2.1\cdot 10^{17}
	> C_{105} \log \alpha.
\]
\end{mystep}

\begin{mystep}\label{step:Step6}
Finally, we combine \eqref{eq:trib:boundStep2} from Step \ref{step:Step2} with \eqref{eq:trib:boundStep6}:
\begin{align*}
	n_1 \leq C_{102}\cdot (\log n_1 \cdot \log b)^2
	&\leq C_{102}\cdot \left(\log n_1 \cdot C_{106} (\log n_1)^3\right)^2\\
	&\leq C_{107} \cdot (\log n_1)^8,
\end{align*}
with
\[
	C_{107}
	= 3.5 \cdot 10^{62}
	> C_{102} \cdot C_{106}^2 =  7.8 \cdot 10^{27} \cdot \left(2.1 \cdot 10^{17}\right)^2.
\]
Solving the inequality $n_1 \leq 3.5 \cdot 10^{62} (\log n_1)^8$ numerically yields
\[
	n_1 \leq 5 \cdot 10^{80}.
\]
\end{mystep}

We have finally found an explicit upper bound for $n_1$. Next we want to reduce this bound. Since the bound for $b$ is extremely large, we cannot use the linear forms from Steps \ref{step:Step1} and \ref{step:Step2} for the reduction process. Instead, we will do four Reduction Steps A--D corresponding to the Steps 3--6 and reduce the bounds as far as we can.

\begin{redstep}[Step \ref{step:Step3}]\label{step:RedstepA}
Recall from \eqref{eq:trib:Lambda3} that
\begin{multline*}
	|\Lambda_3'|
	=|(n_1m_2 - n_2m_1)\log \alpha - (m_1-m_2)\log a|\\
	\leq 20.6 n_1 \max( \alpha^{-(n_1-n_2)}, \alpha^{-(n_2-n_3)},b^{-(m_2-m_3)} ).
\end{multline*}
Note that $m_1-m_2 < m_1 \leq n_1 \leq 5 \cdot 10^{80}$. 
We compute the continued fraction expansion of $\log a/ \log \alpha$ and find the first convergent $p/q$ such that $q\geq 5 \cdot 10^{80}$.
Then by the best approximation property of continued fractions it turns out that
\[
	4.4 \cdot 10^{-82}
	\leq |p \log \alpha - q \log a|
	\leq |\Lambda_3'|.
\]
This implies
\[
	\max( \alpha^{-(n_1-n_2)}, \alpha^{-(n_2-n_3)},b^{-(m_2-m_3)} )
	\geq 4.4 \cdot 10^{-82} \cdot \frac{1}{20.6 n_1}
	\geq 4.2 \cdot 10^{-164}.
\]

\textit{Case 1:} $\max( \alpha^{-(n_1-n_2)}, \alpha^{-(n_2-n_3)},b^{-(m_2-m_3)} )=\alpha^{-(n_1-n_2)}$. Then we get that
\[
	n_1 - n_2
	\leq - \log (4.2 \cdot 10^{-164}) /\log \alpha
	< 618
\]
and we can skip to Reduction Step \ref{step:RedstepC}.

\textit{Case 2:} $\max( \alpha^{-(n_1-n_2)}, \alpha^{-(n_2-n_3)},b^{-(m_2-m_3)} )=\alpha^{-(n_2-n_3)}$. Then we get
\[
	n_2 - n_3
	< 618
\]
and go to the next step.

\textit{Case 3:} $\max( \alpha^{-(n_1-n_2)}, \alpha^{-(n_2-n_3)},b^{-(m_2-m_3)} )=b^{-(m_2-m_3)}$. Then we get
\[
	\log b
	\leq (m_2-m_3)\log b 
	\leq -\log(4.2 \cdot 10^{-164})
	\leq 377
\]
and we can skip to Reduction Step \ref{step:RedstepD}.
\end{redstep}

\begin{redstep}[Step \ref{step:Step4}]\label{step:RedstepB}
Recall from \eqref{eq:trib:Lambda4} that
\begin{multline*}
	|\Lambda_4'|
	= |(m_2-m_1)\log a + (m_2n_1 - m_1n_3) \log \alpha - m_1 \log(\alpha^{n_2-n_3}-1)|\\
	\leq 18.8 n_1 \max( \alpha^{-(n_1-n_2)}, b^{-(m_2-m_3)})
\end{multline*}
and all coefficients are bounded by $n_1^2\leq (5 \cdot 10^{80})^2\leq 2.5\cdot 10^{161}=:M$.

Now for each $n_2-n_3\in \{1,2,\ldots, 617\}$ we apply the LLL-algorithm to find an absolute lower bound for $|\Lambda_4'|$ as described in Lemma \ref{lem:LLL}.
To obtain the matrices $B$ and $B^*$ we use the matrix attributes \verb|LLL()| and \verb|gram_schmidt()| in Sage \cite{sagemath}.
In each case, we try $C\approx M^3$ and if the algorithm fails (i.e.\ $c^2\leq T^2+S$), we increase $C$ by a factor of 10.
Indeed, in each of the cases the LLL reduction works after at most three tries and we get a lower bound for $|\Lambda_4'|$ in that case.
As an overall lower bound we obtain
\[
	3.7 \cdot 10^{-326}
	\leq |\Lambda_4'|.
\]
This implies
\[
	\max( \alpha^{-(n_1-n_2)}, b^{-(m_2-m_3)})
	\geq 3.7 \cdot 10^{-326} \cdot \frac{1}{18.8 n_1}
	\geq 3.9 \cdot 10^{-408}.
\]

\textit{Case 1:} $\max(\alpha^{-(n_1-n_2)}, b^{-(m_2-m_3)}) = \alpha^{n_1-n_2}$: Then
\[
	n_1-n_2 
	\leq -\log (3.9 \cdot 10^{-408})/\log \alpha
	< 1540
\]
and we go to the next step.

\textit{Case 2:} $\max(\alpha^{-(n_1-n_2)}, b^{-(m_2-m_3)}) = b^{-(m_2-m_3)}$: Then
\[
	\log b
	\leq (m_2-m_3)\log b
	\leq - \log (3.9 \cdot 10^{-408})
	\leq 939
\]
and we can skip to Reduction Step \ref{step:RedstepD}.
\end{redstep}

\begin{redstep}[Step \ref{step:Step5}]\label{step:RedstepC}
Now we can compute a small bound for $\log b$ like in Step \ref{step:Step5}:
\[
	\log b 
	\leq (n_1-n_2)\log \alpha + 2.1
	\leq 1539 \cdot \log \alpha + 2.1
	\leq 940.
\]
\end{redstep}

\begin{redstep}[Step \ref{step:Step6}]\label{step:RedstepD}
From \eqref{eq:trib:boundStep2} we get
\[
	n_1 
	\leq 7.8 \cdot 10^{27} (\log b)^2 (\log n_1)^2
	\leq 7.8 \cdot 10^{27} \cdot 940^2 (\log n_1)^2
	\leq 6.9 \cdot 10^{33} \cdot (\log n_1)^2.
\]
Solving this inequality numerically, we obtain
\[
	n_1 \leq 5.3\cdot 10^{37}.
\]
\end{redstep}

\step{step:repeating}{Repeating the reduction steps:}
With this smaller bound for $n_1$ we can now repeat the Reduction Steps A--D (see Table \ref{table:redSteps}). The Sage code is included in the \nameref{sec:appendix}.

\begin{table}[h]
\caption{Repeating the reduction steps}\label{table:redSteps}
\begin{tabular}{rcccc}
\hline
\multicolumn{1}{l}{}                & \textbf{1\ts{st} round} & \textbf{2\ts{nd} round} & \textbf{3\ts{rd} round} & \textbf{4\ts{th} round} \\ \hline
$n_1\leq \ldots$                    & $5 \cdot 10^{80}$       & $5.3\cdot 10^{37}$      & $1.2\cdot 10^{37}$      & $1.1\cdot 10^{37}$      \\ \hline
\multicolumn{1}{l}{\textbf{Step A}} &                         &                         &                         &                         \\
$n_1-n_2\leq \ldots$                & 617                     & 292                     & 288                     & 288                     \\
$n_2-n_3\leq \ldots$                & 617                     & 292                     & 288                     & 288                     \\
$\log b \leq \ldots$                & 377                     & 179                     & 176                     & 176                     \\ \hline
\multicolumn{1}{l}{\textbf{Step B}} &                         &                         &                         &                         \\
$n_1-n_2\leq \ldots$                & 1539                    & 729                     & 719                     & 715                     \\
$\log b \leq \ldots$                & 939                     & 445                     & 439                     & 437                     \\ \hline
\multicolumn{1}{l}{\textbf{Step C}} &                         &                         &                         &                         \\
$\log b \leq \ldots$                & 940                     & 447                     & 441                     & 438                     \\ \hline
\multicolumn{1}{l}{\textbf{Step D}} &                         &                         &                         &                         \\
$n_1\leq \ldots$                    & $5.3\cdot 10^{37}$      & $1.2\cdot 10^{37}$      & $1.1\cdot 10^{37}$      & $1.1\cdot 10^{37}$      \\ \hline
\end{tabular}
\end{table}

After that, we are not able to reduce the bound significantly any further. 
From the bounds in the table one can see that we have proven the bounds from Theorem~\ref{thm:Tribos}, i.e. overall we have proven that under the assumptions of Theorem~\ref{thm:Tribos} we have
\[
	\log b \leq 438
	\quad \text{and} \quad 
	150<n_1\leq 1.1\cdot 10^{37}.
\]

\section*{Appendix}\label{sec:appendix}

Below we have enclosed the Sage code that was used to determine the \hyperlink{step:smallSols}{small solutions} to $T_{n_1}-T_{n_2}=b^{m_1}-b^{m_2}$ in Section \ref{sec:Tribos}, as well as the code that was used to determine the bounds for the reduction rounds in Table \ref{table:redSteps}. 

\begin{lstlisting}[language=Python]
# small solutions

# start computing Tribonacci numbers T_n1
t1_veryold = 0; t1_old = 0; t1 = 1 # starting values
n1 = 1
while n1 < 150:
    # compute next Tribonacci number T_n1
    temp = t1 + t1_old + t1_veryold 
    t1_veryold = t1_old
    t1_old = t1
    t1 = temp
    n1 = n1 + 1
    # start computing Tribonacci numbers T_n2 < T_n1
    t2_veryold = 0; t2_old = 0; t2 = 1 # starting values
    n2 = 1
    while t2 < t1_old: # because we increase at the beginning
        # compute next Tribonacci number T_n2
        temp = t2 + t2_old + t2_veryold
        t2_veryold = t2_old
        t2_old = t2
        t2 = temp
        n2 = n2 + 1
        # check representation
        diff = t1 - t2
        primefactorisation = list(factor(diff))
        factors_for_b = list(Combinations(primefactorisation))
        del factors_for_b[0] # exclude b = 1
        for factors in factors_for_b:
            x = gcd([k for [p,k] in factors])
            b = prod(p^(k/x) for [p,k] in factors)
            y = round(log(1 + diff/(b^x))/log(b))
            if b^x * (b^y - 1) == diff:
                m2 = x
                m1 = x + y
                c = t2 - b^x
                print("T_{",n1,"} - T_{",n2,"} 
                      =",b,"^{",m2,"} (",b,"^{",y,"}-1), \\quad 
                      c&=", c, "= T_{",n1,"} -",b,"^{",m1,"} 
                      = T_{",n2,"} -",b,"^{",m2,"}; \\\\")  
\end{lstlisting}

\begin{lstlisting}[language=Python]
# reduction steps

alpha = n(solve(x^3 - x^2 - x - 1 == 0, x, 
          solution_dict=True)[2][x], digits=2000)
a = 1/(-alpha^2 + 4*alpha - 1)

#################################################################
print("Step A:")

bound_n1 = 5.3*10^37  # replace for each round

c = continued_fraction(log(a)/log(alpha))
i = 1
while c.denominator(i) < bound_n1:
    i = i + 1
    
p = c.numerator(i)
q = c.denominator(i)

lowerbound_linform = abs(p*log(alpha) - q*log(a))

lowerbound = lowerbound_linform/(20.6*bound_n1) 

n2n3max = floor(-log(lowerbound)/log(alpha))
print("bound n_1 - n_2 and n_2 - n_3:", n2n3max)
print("bound log(b):",  -log(lowerbound))

#################################################################
print("Step B:")

M = bound_n1^2

C0 = 10^int(3*log(M)/log(10)) # approx. M^3 but with full precision

lowerBound = 1

for n2n3 in range(1, n2n3max+1):
	# loop in order to find a C that works    
    C = C0
    Cmax = C0*10^30
    done = False
    while not done:
        C = C*10
    
        A = Matrix([[1, 0, round(C*log(a))],
                  [0, 1, round(C*log(alpha))],
                  [0, 0, round(C*log(alpha^n2n3 - 1))]])
        B = A.LLL()
        Bstar, mu = B.gram_schmidt()

        c = min([norm(N(b)) for b in Bstar])
        S = 2*M^2
        T = (1 + 3*M)/2

        if c^2 > T^2 + S:
            lowerBound = min(lowerBound, 1/C * (sqrt(c^2 - S) - T))
            done = True
        elif C == Cmax:
            print('did not work for n2n3 =', n2n3)
            break           
            
lowerBound2 = lowerBound/(18.8*bound_n1)
n1n2max = floor(-log(lowerBound2)/log(alpha))
print("bound n_1 - n_2:", n1n2max)
print("bound log b:", -log(lowerBound2))

#################################################################
print("Step C:")

logbmax = n1n2max*log(alpha) + 2.1
print("bound log b:", logbmax.n())

#################################################################
print("Step D:")

logbmax = max(logbmax, -log(lowerBound2))
newbound_n1 = find_root(x - 7.8*10^27*logbmax^2 * log(x)^2, 
                        7.8*10^27*logbmax^2, (7.8*10^27*logbmax^2)^2)
print("new bound for n_1:", newbound_n1)
\end{lstlisting}

\bibliographystyle{habbrv}
\bibliography{lit_Sebastian}

\begin{thebibliography}{10}
\expandafter\ifx\csname url\endcsname\relax
  \def\url#1{\texttt{#1}}\fi
\expandafter\ifx\csname doi\endcsname\relax
  \def\doi#1{\burlalt{doi:#1}{http://dx.doi.org/#1}}\fi
\expandafter\ifx\csname urlprefix\endcsname\relax\def\urlprefix{URL: }\fi
\expandafter\ifx\csname href\endcsname\relax
  \def\href#1#2{#2}\fi
\expandafter\ifx\csname burlalt\endcsname\relax
  \def\burlalt#1#2{\href{#2}{#1}}\fi

\bibitem{bawu93}
A.~Baker and G.~W\"{u}stholz.
\newblock Logarithmic forms and group varieties.
\newblock {\em J. Reine Angew. Math.}, 442:19--62, 1993.
\newblock \doi{10.1515/crll.1993.442.19}.

\bibitem{BatteEtAl2022}
H.~Batte, M.~Ddamulira, J.~Kasozi, and F.~Luca.
\newblock On the multiplicity in {P}illai's problem with {F}ibonacci numbers
  and powers of a fixed prime.
\newblock \href{https://arxiv.org/abs/2207.12868}{arXiv:2207.12868}.

\bibitem{Bennett2001}
M.~A. Bennett.
\newblock On some exponential equations of {S}. {S}. {P}illai.
\newblock {\em Canad. J. Math.}, 53(5):897--922, 2001.
\newblock \doi{10.4153/CJM-2001-036-6}.

\bibitem{BravoDiazGomez2021}
J.~J. Bravo, M.~D\'{\i}az, and C.~A. G\'{o}mez.
\newblock Pillai's problem with {$k$}-{F}ibonacci and {P}ell numbers.
\newblock {\em J. Difference Equ. Appl.}, 27(10):1434--1455, 2021.
\newblock \doi{10.1080/10236198.2021.1990900}.

\bibitem{BravoLucaYazan2017}
J.~J. Bravo, F.~Luca, and K.~Yaz\'{a}n.
\newblock On {P}illai's problem with {T}ribonacci numbers and powers of 2.
\newblock {\em Bull. Korean Math. Soc.}, 54(3):1069--1080, 2017.
\newblock \doi{10.4134/BKMS.b160486}.

\bibitem{ChimPinkZiegler2017}
K.~C. Chim, I.~Pink, and V.~Ziegler.
\newblock On a variant of {P}illai's problem.
\newblock {\em Int. J. Number Theory}, 13(7):1711--1727, 2017.
\newblock \doi{10.1142/S1793042117500981}.

\bibitem{ChimPinkZiegler2018}
K.~C. Chim, I.~Pink, and V.~Ziegler.
\newblock On a variant of {P}illai's problem {II}.
\newblock {\em J. Number Theory}, 183:269--290, 2018.
\newblock \doi{10.1016/j.jnt.2017.07.016}.

\bibitem{Ddamulira2019}
M.~Ddamulira.
\newblock On the problem of {P}illai with {P}adovan numbers and powers of 3.
\newblock {\em Studia Sci. Math. Hungar.}, 56(3):364--379, 2019.
\newblock \doi{10.1556/012.2019.56.3.1435}.

\bibitem{Ddamulira2019Tribos}
M.~Ddamulira.
\newblock On the problem of {P}illai with tribonacci numbers and powers of 3.
\newblock {\em J. Integer Seq.}, 22(5):Art. 19.5.6, 14, 2019.
\newblock
  \urlprefix\url{https://www.emis.de/journals/JIS/VOL22/Ddamulira/dda3.pdf}.

\bibitem{Ddamulira2020}
M.~Ddamulira.
\newblock On a problem of {P}illai with {F}ibonacci numbers and powers of 3.
\newblock {\em Bol. Soc. Mat. Mex. (3)}, 26(2):263--277, 2020.
\newblock \doi{10.1007/s40590-019-00263-1}.

\bibitem{DdamuliraGomezCarlosLuca2018}
M.~Ddamulira, C.~A. G\'{o}mez, and F.~Luca.
\newblock On a problem of {P}illai with {$k$}-generalized {F}ibonacci numbers
  and powers of 2.
\newblock {\em Monatsh. Math.}, 187(4):635--664, 2018.
\newblock \doi{10.1007/s00605-018-1155-1}.

\bibitem{DdamuliraLuca2020}
M.~Ddamulira and F.~Luca.
\newblock On the problem of {P}illai with {$k$}-generalized {F}ibonacci numbers
  and powers of 3.
\newblock {\em Int. J. Number Theory}, 16(7):1643--1666, 2020.
\newblock \doi{10.1142/S1793042120500876}.

\bibitem{DdamuliraLucaRakotomalala2017}
M.~Ddamulira, F.~Luca, and M.~Rakotomalala.
\newblock On a problem of {P}illai with {F}ibonacci numbers and powers of 2.
\newblock {\em Proc. Indian Acad. Sci. Math. Sci.}, 127(3):411--421, 2017.
\newblock \doi{10.1007/s12044-017-0338-3}.

\bibitem{ErazoGomezLuca2021}
H.~S. Erazo, C.~A. G\'{o}mez, and F.~Luca.
\newblock On {P}illai's problem with {$X$}-coordinates of {P}ell equations and
  powers of 2 {II}.
\newblock {\em Int. J. Number Theory}, 17(10):2251--2277, 2021.
\newblock \doi{10.1142/S1793042121500871}.

\bibitem{HernandezLucaRivera2019}
S.~H. Hern\'{a}ndez, F.~Luca, and L.~M. Rivera.
\newblock On {P}illai's problem with the {F}ibonacci and {P}ell sequences.
\newblock {\em Bol. Soc. Mat. Mex. (3)}, 25(3):495--507, 2019.
\newblock \doi{10.1007/s40590-018-0223-9}.

\bibitem{HernaneLucaRihaneTogbe2018}
M.~O. Hernane, F.~Luca, S.~E. Rihane, and A.~Togb\'{e}.
\newblock On {P}illai's problem with {P}ell numbers and powers of 2.
\newblock {\em Hardy-Ramanujan J.}, 41:22--31, 2018.
\newblock \urlprefix\url{https://hrj.episciences.org/5102/pdf}.

\bibitem{Laurent2008}
M.~Laurent.
\newblock Linear forms in two logarithms and interpolation determinants. {II}.
\newblock {\em Acta Arith.}, 133(4):325--348, 2008.
\newblock \doi{10.4064/aa133-4-3}.

\bibitem{LomeliHernandezLuca2019}
A.~C.~G. Lomel\'{\i}, S.~H. Hern\'{a}ndez, and F.~Luca.
\newblock Pillai's problem with the {F}ibonacci and {P}adovan sequences.
\newblock {\em Ann. Math. Inform.}, 50:101--115, 2019.
\newblock \doi{10.33039/ami.2019.09.001}.

\bibitem{LomeliHernandez2019}
A.~C.~G. Lomel\'{\i} and S.~Hern\'{a}ndez~Hern\'{a}ndez.
\newblock Pillai's problem with {P}adovan numbers and powers of two.
\newblock {\em Rev. Colombiana Mat.}, 53(1):1--14, 2019.
\newblock
  \urlprefix\url{http://www.scielo.org.co/pdf/rcm/v53n1/0034-7426-rcm-53-01-1.pdf}.

\bibitem{LomeliHernandezLuca2019Indian}
A.~C.~G. Lomel\'{\i}, S.~Hern\'{a}ndez~Hern\'{a}ndez, and F.~Luca.
\newblock Pillai's problem with the {P}adovan and tribonacci sequences.
\newblock {\em Indian J. Math.}, 61(1):61--75, 2019.

\bibitem{Matveev2000}
E.~M. Matveev.
\newblock An explicit lower bound for a homogeneous rational linear form in the
  logarithms of algebraic numbers. {II}.
\newblock {\em Izv. Math.}, 64(6):1217--1269, 2000.
\newblock \doi{10.1070/im2000v064n06abeh000314}.

\bibitem{Mignotte:kit}
M.~Mignotte.
\newblock A kit on linear forms in three logarithms.
\newblock preprint.

\bibitem{Pillai1936}
S.~S. Pillai.
\newblock On $a^x-b^y=c$.
\newblock {\em Indian Math. Soc.}, 2:119--122, 1936.

\bibitem{Pillai1937}
S.~S. Pillai.
\newblock A correction to the paper ``{O}n $a^x-b^y=c$''.
\newblock {\em Indian Math. Soc.}, 2:215, 1937.

\bibitem{Schinzel1974}
A.~Schinzel.
\newblock Primitive divisors of the expression {$A^{n}-B^{n}$} in algebraic
  number fields.
\newblock {\em J. Reine Angew. Math.}, 268(269):27--33, 1974.
\newblock \doi{10.1515/crll.1974.268-269.27}.

\bibitem{Smart1998}
N.~P. Smart.
\newblock {\em {The Algorithmic Resolution of Diophantine Equations}}.
\newblock London Mathematical Society Studen Texts 41. Cambridge University
  Press, 1998.
\newblock \doi{10.1017/CBO9781107359994}.

\bibitem{sagemath}
{The Sage Developers}.
\newblock {\em {S}ageMath, the {S}age {M}athematics {S}oftware {S}ystem
  ({V}ersion 9.2)}, 2021.
\newblock \urlprefix\url{https://www.sagemath.org}.

\bibitem{wu88}
G.~W\"{u}stholz.
\newblock A new approach to {B}aker's theorem on linear forms in logarithms.
  {III}.
\newblock In {\em New advances in transcendence theory ({D}urham, 1986)}, pages
  399--410. Cambridge Univ. Press, Cambridge, 1988.

\bibitem{Zannier2009}
U.~Zannier.
\newblock {\em Lecture Notes on Diophantine Analysis}.
\newblock Scuola Normale Superiore, 2009.

\end{thebibliography}
	
\end{document}